\documentclass[10pt,reqno]{amsart}

\usepackage[T1]{fontenc}
\usepackage[utf8]{inputenc}
\usepackage{lmodern}
\usepackage{amsmath,amssymb,amsthm,mathtools,mathrsfs}
\usepackage{cite}
\usepackage{enumerate}
\usepackage{microtype}
\usepackage[
colorlinks=true,
linkcolor=blue,
citecolor=blue,
urlcolor=blue,
]{hyperref}
\numberwithin{equation}{section}

\newtheorem{theorem}{Theorem}[section]
\newtheorem{proposition}[theorem]{Proposition}
\newtheorem{lemma}[theorem]{Lemma}
\newtheorem{corollary}[theorem]{Corollary}
\newtheorem{assumption}[theorem]{Assumption}
\newtheorem{definition}[theorem]{Definition}
\theoremstyle{remark}
\newtheorem{remark}[theorem]{Remark}
\theoremstyle{plain}

\newcommand{\real}{{\mathbb R}}
\newcommand{\nat}{{\mathbb N}}
\newcommand{\ent}{{\mathbb Z}}
\newcommand{\com}{{\mathbb C}}

\newcommand{\M}{\mathcal M}
\newcommand{\B}{\mathcal B}
\newcommand{\N}{\mathcal N}
\newcommand{\Aalg}{\mathcal A}
\newcommand{\Eform}{\mathcal E}
\renewcommand{\H}{\mathcal H}
\newcommand{\Ent}{\operatorname{Ent}}

\newcommand{\Dom}{\operatorname{Dom}}
\newcommand{\sgn}{\operatorname{sgn}}
\newcommand{\id}{\operatorname{id}}
\newcommand{\tr}{\operatorname{tr}}
\newcommand{\dd}{\,\mathrm d}
\newcommand{\1}{\mathbf 1}

\renewcommand\d{{\rm d}}

\newcommand\jdz[1]{ \left| #1 \right| }
\newcommand\sk[1]{ \left( #1 \right) }
\newcommand\mk[1]{ \left[ #1 \right] }
\newcommand\lk[1]{ \left\{ #1 \right\} }
\newcommand{\lara}[1]{\left\langle #1 \right\rangle}

\title[Noncommutative $L_p$-Poincar\'{e} inequalities]
{Noncommutative $L_p$-Poincar\'{e} inequalities}

\author[Li Gao]{Li Gao}
\address{School of Mathematics and Statistics, Wuhan University, Hubei Province 430072, P. R. China}
\email{gao.li@whu.edu.cn}

\author[Marius Junge]{Marius Junge}
\address{Department of Mathematics, University of Illinois, Urbana, IL, 61801, USA}
\email{mjunge@illinois.edu}

\author[Zhendong  Xu]{Zhendong Xu}
\address{%Institute for Advanced Study in Mathematics, Harbin Institute of Technology,  Harbin 150001, China; and
Department of Mathematical Sciences and the Research Institute of Mathematics, Seoul National University, Gwanak-ro 1, Gwanak-gu, Seoul 08826, Republic of Korea}
\email{zhendong\_xu\_97@snu.ac.kr}

\author[Hao Zhang]{Hao Zhang}
\address{%Institute for Advanced Study in Mathematics, Harbin Institute of Technology,  Harbin 150001, China; and
Instituto de Ciencias Matem\'{a}ticas, Consejo Superior de Investigaciones Cient\'{i}ficas, C/ Nicol\'{a}s Cabrera 13-15, 28049, Madrid, Spain}
\email{hao.zhang@icmat.es}
\date{}

\begin{document}

\maketitle

\begin{abstract}
We derive $L_p$-Poincar\'{e} inequalities for quantum Markov semigroups on finite von Neumann algebras, from a spectral gap (PI) or from a modified logarithmic Sobolev inequality (MLSI). For $p\ge2$, the corresponding coefficients have optimal order $O(p)$ and $O(\sqrt p)$, respectively. Both estimates follow from Beckner inequalities by integration along semigroups. We also establish a non-tracial extension to quantum Markov semigroups satisfying GNS detailed balance. As a direct byproduct, we apply $L_p$-Poincar\'{e} inequalities to Schatten-valued functions on graphs in order to obtain the Schatten distortion of expanders.
\end{abstract}

\tableofcontents

\section{Introduction}

The $L_p$-Poincar\'{e} inequalities are basic tools for moment and concentration estimates in probability (cf. \cite{Bakry1996, Talagrand1991, Talagrand1995, Bobkov1997, Ledoux2001,AdamczakWolff2015} etc.). Their metric counterparts are used to study nonlinear spectral gaps and embedding obstructions in Banach space geometry (cf. \cite{MendelNaor2014,Naor2014,IVHV}). A classical example is Pisier's inequality for the Gaussian measure \cite{Pisier1986}. Let $\gamma_n$ be the standard Gaussian probability measure on
$\mathbb R^n$, then
\[
	\left\| f - \int_{\real^n} f\,\d \gamma_n \right\|_p \leq C \sqrt{p} \| \nabla(f) \|_p,
\]
for all $p\geq 2$ and all differentiable functions $f$ with $\nabla(f) \in L_p(\real^n; \gamma_n)$. The constant $O(\sqrt{p})$ is crucial for applications in concentration inequalities, as revealed in \cite{Bobkov1999, Efraim2008,Junge2015_2}. 

The corresponding question for quantum Markov semigroups is which assumptions yield such $L_p$ estimates, and with what dependence on $p$. Junge and Zeng obtained $L_p$-Poincar\'{e} inequalities for noncommutative diffusion semigroups under the curvature condition (often called as $\Gamma^2$-condition \cite{Bakry1985}), with coefficients growing in $O(p)$. In particular, they also obtained the $O(\sqrt{p})$ growth when the gradient is taken with respect to the $L_\infty$-norm. To recover the form of Pisier's estimate, one would like to retain the $L_p$-norm for gradients. This is known for some special semigroups associated with group cocycles under the $\Gamma^2$-condition \cite[Theorem~1.2]{Junge2015_2}.

\iffalse
Quantum analogues of these inequalities are essential for understanding noncommutative spaces. In particular, quantum Markov semigroups have been widely employed in operator algebras for studying approximation properties, structure theory, and noncommutative analysis (see e.g., \cite{Caspers2021,Junge2013}). Significant contributions include Efraim and Lust-Piquard's remarkable work \cite{Efraim2008} of $L_q$-Poincar\'{e}-type inequalities with $O(\sqrt{q})$ constants for discrete cubes and CAR algebras. Later, Junge and Zeng \cite{Zeng2014, Junge2015,Junge2015_2} developed a comprehensive and quite general framework centered on the noncommutative curvature dimension condition (often called as $\Gamma^2$-condition). They established that this $\Gamma^2$-condition, the quantum counterpart of the classical Bakry-\'{E}mery condition \cite{Bakry1985}, implies $L_q$-Poincar\'{e}-type inequalities, and they also utilized $L_q$-Poincar\'{e}-type inequalities to derive concentration estimates and transportation inequalities. Nevertheless, their derived estimates of best constants were not optimal for some concrete but important examples of quantum Markov semigroups. Motivated by this, we aim to improve the estimates of best constants on $L_q$-Poincar\'{e}-type inequalities.
\fi

In this paper, we aim to obtain the $O(\sqrt{p})$ growth with the $L_p$-norm for gradients from a modified logarithmic Sobolev inequality in the general setting. We also obtain the $O(p)$ growth coefficient from a spectral gap, with a different proof from \cite{JW}.

Now we illustrate our main results explicitly. Let $(\M, \tau)$ be a finite von Neumann algebra equipped with a normal faithful tracial state $\tau$. Let $(T_t)_{t \geq 0}$ be a quantum Markov semigroup and $A$ be its infinitesimal generator on $L_2(\M, \tau)$, and see Section \ref{sec:setting} for details. Denote by $\mathcal{M}^{T} = \{ x \in \mathcal{M} : T_t(x) = x \text{ for all } t \geq 0 \}$ the fixed-point algebra, and $E : \M \to \M^T$ the conditional expectation. 

At first, we recall some key definitions. Define $\Eform(x, y) = \tau(x^*Ay)$ as the Dirichlet form. We say that $(T_t)$ has a spectral gap  $\lambda>0$ relative to $E$ if
\begin{equation}
 \lambda\|x-E(x)\|_2^2\leq\Eform(x,x), \quad x \in {\rm dom}(A^{1/2}). \tag{PI}
 \label{eq:spectral-gap}
\end{equation}
The largest possible constant $\lambda$ is called the \emph{spectral gap}.

For positive invertible $x\in\M$, define the \emph{relative entropy}, with the convention $0\log0=0$, by
\begin{equation}
 \Ent_E(x)
 =\tau(x\log x)-\tau(E(x) \log E(x)). 
 \label{eq:conditional-entropy}
\end{equation}
For noninvertible $x$, this
identity remains valid in the regularized sense obtained from
$\log(E(x)+\varepsilon\1)$.  We say that $(T_t)$ satisfies a modified logarithm Sobolev inequality 
($\mathrm{MLSI}$) if
\begin{equation}
 \alpha\Ent_E(x)\leq\Eform(x,\log x) \tag{MLSI}
 \label{eq:mlsi}
\end{equation}
for every positive invertible $x \in {\rm dom}(A^{1/2})$ with $\log x \in {\rm dom}(A^{1/2})$. The largest possible constant $\alpha$ is called the MLSI constant.
\iffalse
If the
definition is initially imposed only on densities $\tau(h)=1$, then
\eqref{eq:mlsi} holds for arbitrary positive invertible $h$: for $c>0$,
\begin{equation}
 \Ent_E(ch)=c\Ent_E(h),
 \qquad
 \Eform(ch,\log(ch))=c\Eform(h,\log h),
 \label{eq:mlsi-homogeneity}
\end{equation}
because the affine $c\log c$ terms cancel and $\Eform(h,\1)=0$.
\fi

For $x, y \in {\rm dom}(A)$ with $x^*y \in {\rm dom}(A)$, let $\Gamma$ be the carr\'{e} du champ associated with the generator $A$:
\[
	\Gamma(x, y) = \frac{1}{2} \sk{ (Ax)^* y + x^* Ay - A(x^* y) }.
\]
%{\color{red}$(T_t)$ is said to be a noncommutative diffusion if $\Gamma(x, x) \in L_1(\M)$ for all $x \in {\rm dom}(A^{1/2}) \cap \M$.}
Our main result, which will be proved in Section \ref{sec:PI-MLSI-to-p}, is the following.
\begin{theorem}\label{Main-p-PI}
	Let $(T_t)$ be a quantum Markov semigroup on a finite von Neumann algebra $\M$. %{\color{blue}Assume $(T_t)$ is an nc diffusion}.
       Suppose that \eqref{eq:spectral-gap} holds with constant $\lambda > 0$, then for all $p \geq 2$ and any selfadjoint $x\in\M$ with $x, x^2 \in {\rm dom}(A)$,
	\begin{equation}\label{main-1}
		\| x - E(x) \|_p \leq \frac{ \sqrt{2}p }{\sqrt{\lambda}} \| \Gamma(x, x)^{1/2} \|_{p}. 
	\end{equation}
	Suppose that \eqref{eq:mlsi} holds with constant $\alpha > 0$, then
	\begin{equation}\label{main-2}
		\| x - E(x) \|_p \leq \frac{ 2\sqrt{2p} }{\sqrt{\alpha} } \| \Gamma(x, x)^{1/2} \|_{p}.
	\end{equation}
\end{theorem}

The above inequalities are understood in the extended sense when the gradient norm is infinite. A more general version for arbitrary elements (not necessarily self-adjoint), involving the gradients of both $x$ and $x^*$, follows from Theorem~\ref{thm:non-selfadjoint}.

The proof has two parts. We first show that a Beckner inequality at the exponent $q = p/(p - 1)$ with explicit constant implies the $L_p$-Poincar\'{e} inequality with corresponding coefficient dependence.
This follows from Lemma \ref{new-lemma3.2} and integration along heat flow. 
Second, PI and MLSI yield the required Beckner constants. In particular, the latter implication follows from a positive integral representation of relative entropy.

In the classical setting, the comparison between MLSI and Beckner inequalities was studied by Adamczak, Polaczyk and Strzelecki \cite{Adamczak2022}. Our approach is different from theirs and we obtain the best universal comparison constants; this answers the question posed in \cite[Remark~2.10]{Adamczak2022} in the scalar-fixed-point reversible setting. 
% The lower bound follows from an exact positive superposition of shifted entropy functions and does not use the moment induction of that work. A signed hyperbolic superposition then yields a conditional Beckner profile with coercivity $\rho_q\to1$ as $q\uparrow2$.

Earlier study on noncommutative $L_p$-Poincar\'e inequalities include Efraim and Lust-Piquard's remarkable work \cite{Efraim2008} for discrete cubes and CAR algebras with $O(\sqrt{q})$ coefficients. Junge and Zeng \cite{Zeng2014, Junge2015,Junge2015_2} devoloped general framework by a martingale approach under $\Gamma^2$ condition. Matrix-valued Poincar\'e inequalities are developed in \cite{HT,HuangTropp2021}. Li and Lu also obtained $L_p$-Poincar\'e-type inequalities from
Beckner inequalities for finite-dimensional matrix algebras \cite[Proposition 3.55]{Li2023}, but their
corresponding constants depends on the dimension, see \cite[Remark 3.22]{Li2023}. Jiao \emph{et al.} \cite[Corollary B]{Jiao2024} obtained Poincar\'{e}-type inequalities for noncommutative Orlicz spaces as byproducts of their main theorem; however, their approach requires stronger assumptions including hypercontractivity and  boundedness of Riesz transform. They did not even analyze the behaviors of best constants. Quite recently, for the spectral gap hypothesis, Junge and Wang \cite{JW} obtained the constant $p/\sqrt{2\lambda}$ for $p \geq 3$, treated $2<p<3$ with one additional factor $\sqrt{2}$, and established a nontracial extension for the cases of GNS detailed balance. We would like to emphasize that our argument applies to both PI and MLSI, and the best constant $O(\sqrt{p})$ provided by the latter. 
%Appendix \ref{sec:appendix} will give a further reinement of the coefficient from PI.

This paper is organized as follows. Section~\ref{sec:setting} is devoted to the preliminaries on noncommutatative $L_p$ spaces and quantum Markov semigroups, and we also devolope a finite-time dimodule structrue used in the proof.  In Section~\ref{sec:heat-flow}, we derive $L_p$-Poincar\'{e} inequalities from Beckner inequalities. In Section~\ref{sec:PI-MLSI-to-p}, we obtain the required Beckner constants from PI or MLSI, completing the proof of Theorem~\ref{Main-p-PI}. We deal with the non-tracial extension in Section~\ref{sec:extensions}. Section~\ref{sec:schatten-graphs} concerns the Schatten-valued Poinca\'{e} inequalities on expanders and the resulting distorsion estimate. Finally, in Section~\ref{sec:example}, we present many concrete examples of quantum Markov semigroups, and apply the main theorem to them.

{\bf Notation.} We use the following notation throughout the paper: for a $*$-algebra $\B$, we denote by $\B_+$ the positive part and $\B_{\rm sa}$ the self-adjoint part.

\bigskip

\section{Preliminaries}
\label{sec:setting}

\subsection{Noncommutative $L_p$ space}
We recall fundamental concepts from von Neumann algebra theory \cite{Takesaki1979}. Let $\M$ be a finite von Neumann algebra of operators acting on a Hilbert space $H$, equipped with a normalized normal faithful trace $\tau$. 
%We denote by $\M_+$ the positive part of $\M$, i.e., the set of positive operators in $\M$. For $x \in \M$, we define the support of $x$, denoted by $\supp\, x$, as the least projection $e \in \M$ such that $x e = x$ or equivalently, $e x = x$. Denote $ \S_+ = \lk{ x \in \M_+ : \tau ( \supp\, x ) < \infty }$, let $S$ be the linear span of $\S_+$. Then $\S$ is a $*$-subalgebra of $\M$ which is weak-$*$ dense in $\M$. Moreover, for $0 < p < \infty$ and $x \in \S$, we have $|x|^p \in \S_+$, where $|x| = (x^* x)^{ \frac{1}{2} }$, thus $\tau (|x|^p) < \infty$. We define
%\[
%	\| x \|_p = \mk{ \tau \sk{ |x|^p } }^{ \frac{1}{p} }, \quad x \in \S.
%\]
%Then $\| \cdot \|_p$ becomes a norm on $\S$ if $1 \leq p < \infty$, and a quasi-norm if $0 < p < 1$. Let $L_p (\M, \tau)$ (or simply $L_p (\M)$) denote the completion of $(\S, \| \cdot \|_p)$. We call this the noncommutative $L_p$ space associated with $(\M, \tau)$. In particular, the elements of $L_p(\M)$ can be described by densely defined closed operators measurable with respect to $(\M, \tau)$. We set $L_\infty (\M) = \M$ equipped with the operator norm for convenience. $\tau$ extends to a linear functional on $\S$, and since $|\tau(x)| \leq \| x \|_1$ for all $x \in \S$, it further extends to a linear functional on $L_1(\M)$.
A densely defined closed operator $x$ affiliated to $\M$ is \emph{measurable} if for every $\varepsilon > 0$, there exists a projection $e$ such that 
\[
	e(H) \subseteq {\rm dom}(x), \quad \tau(1 - e) < \varepsilon,
\]
where ${\rm dom}(x)$ denotes the domain of $x$. It is clear that any operator $x$ affiliated to $\M$ is measurable if $\tau$ is finite. The space $L_0(\mathcal{M})$ consists of all such measurable operators, and $\tau$ extends uniquely to a linear functional on the positive cone $L_0(\mathcal{M})_+$.

\iffalse
We shall introduce the noncommutative distribution function and generalized singular number, which will be useful in the context of the generalized $L_p$ spaces. Let $x \in L_0(\M)$, then $|x|$ has spectral measure $\e_s(|x|)$, denote $e^\bot_s (|x|) = \1_{s, \infty}(|x|)$ the spectral projection of $x$ corresponding to the interval $(s, \infty)$. Then we define
\[
	\lambda_s(x) = \tau \sk{ e^\bot_s(|x|) }, \ s> 0; \quad \mu_t(x) = \inf \lk{ s: \lambda_s(x) \leq t }, \ t > 0.
\]  
The function $s \mapsto \lambda_s(x)$ is called the distribution function of $x$ and $t \mapsto \mu_t(x)$ the generalized singular numbers of $x$.

\begin{proposition}\label{distribution_x}
	Let $x \in L_0(\M)_+$, then
	\begin{enumerate}[(i)]
		\item For any non-decreasing non-negative continuous function $\varphi$ on $\real_+$, we have
		\[	
			\tau\sk{ \varphi(x) } = \int_0^\infty \varphi ( \mu_t(x) ) \, \d t.
		\] 
		\item (Ky Fan inequality) For any non-decreasing convex function $\varphi$ on $\real_+$ with $\varphi(0) = 0$, we have
		\[
			\int_0^t \varphi ( \mu_s(x + y) ) \, \d s \leq \int_0^t \varphi ( \mu_s(x) + \mu_s(y) ) \, \d s.
		\]
		\item For any non-decreasing convex function $\varphi$ on $\real_+$ with $\varphi(0) = 0$ such that $t \mapsto \varphi(e^t)$ is convex on $\real$, we have
		\[
			\int_0^t \varphi ( \mu_s(xy) ) \, \d s \leq \int_0^t \varphi \sk{ \mu_s(x) \mu_s(y) } \, \d s.
		\]
	\end{enumerate}
\end{proposition}
\fi

Recall that every selfadjoint operator $x \in L_0(\M)_{\rm sa}$ admits a unique spectral decomposition
\[
	x = \int_\mathbb{R} \lambda \, \d e_\lambda(x),
\]
where $(e_\lambda(x))_\lambda$ is the spectral measure of $x$, supported on the spectrum $\sigma(x)$. Let $\varphi$ be a Borel function on $\real$, and the densely defined closed operator $\varphi(x)$ is defined via the following spectral integral
\[	
	\varphi(x) = \int_\mathbb{R} \varphi(\lambda) \, \d e_\lambda(x).
\] 
Then for $0<p<\infty$, the $L_p$-(quasi)norm is defined as
\[
	\| x \|_p = \mk{ \tau \sk{ |x|^p } }^{ \frac{1}{p} }, \quad \text{where} \quad |x| = (x^* x)^{1/2}, \ x \in L_0(\M).
\]
This is a norm for $1 \leq p < \infty$ and a quasi-norm for $0 < p < 1$. The noncommutative $L_p$ space, denoted by $L_p(\mathcal{M})$, consists of all $x \in L_0(\mathcal{M})$ with $\| x \|_p < \infty$. We refer the interested reader to \cite{Xu2007book} for more information on noncommutative $L_p$ spaces. 

For any $x \in L_0(\M)_{\rm sa}$, we deonte by
\[
x_+ = \1_{[0, \infty)}(x)x, \quad x_- = -\1_{(-\infty, 0)}(x)x, \quad J_p(x) = \sgn(x)|x|^{p-1}.
\]
The following identities are direct and will be used later.

\begin{lemma}
\label{lem:dual-seed}
Let $p>1$ and $p'=p/(p-1)$, and let $x \in L_p(\M)_{\rm sa}$. Then
\begin{equation}
 J_p(x) \in L_{p'}(\M),
 \quad
 \|J_p(x)\|_{p'}=\|x\|_p^{p-1},
 \label{eq:dual-seed-norm}
\end{equation}
and
\begin{equation}
J_p(x)x \in L_1(\M), \quad \tau(J_p(x)x) = \|x\|_p^p.
 \label{eq:dual-seed-pairing}
\end{equation}
% If $y\in\Aalg_{\rm sa}$, then the standing core hypothesis places $y_\pm^{p-1}$ in the bounded functional-calculus core used below.
\end{lemma}

\begin{proof}
This lemma follows immediately from
\[
 |J_p(x)|^{p'} = \int_\real |\sgn(\lambda) |\lambda|^{p - 1} |^{p'} \,\d e_\lambda(x) = \int_\real |\lambda|^{p}\,\d e_\lambda(x) = |x|^p 
\]
and
\[
 J_p(x)x
 = \int_\real \sgn(\lambda) |\lambda|^{p - 1} \lambda \,\d e_\lambda(x) = \int_\real |\lambda|^{p}\,\d e_\lambda(x) = |x|^p,
\]
as desired.
\end{proof}

\subsection{Quantum Markov semigroups}
We present fundamental properties of quantum Markov semigroups in the tracial setting, see e.g. \cite{Junge2010,Brannan2022}.
\begin{definition}
	A family of linear maps $(T_t)_{t \geq 0}$ on a finite von Neumann algebra $\M$ is a quantum Morkov semigroup if
	\begin{enumerate}[(i)]
		\item Every $T_t$ is a normal completely positive map on $\M$ such that $T(\1) = \1$;
		\item Every $T_t$ is selfadjoint with respect to the trace $\tau$, namely, $\tau ( T_t(x)^* y ) = \tau ( x^* T_t(y) )$;
		\item $T_s \circ T_t = T_{s+t}$ for all $s,t \geq 0$ with $T_0 = {\rm id}$;
		\item For each $x \in \M$, the map $t \mapsto T_t(x)$ is continuous  with respect to the $w^\ast$-topology on $\M$.
	\end{enumerate} 
\end{definition}
 
For simplicity, we just write $(T_t)$ to denote the quantum Markov semigroup $(T_t)_{t\geq 0}$. Conditions (i) and (ii) imply that each $T_t$ is trace preserving. Indeed, $\tau(T_t(x)) = \tau(x)$ for all $x \in \mathcal{M}_+$. Consequently, each $T_t$ extends uniquely to a contraction on $L_p(\mathcal{M})$ for $1 \leq p \leq \infty$ (cf. \cite{Junge2007}). The infinitesimal generator of $(T_t)$, denoted by $A$, is defined as
	\[
	A(x) = \lim_{t \to 0_+} \frac{x - T_t(x)}{t},
	\]
	where $x \in {\rm dom}(A)$, and ${\rm dom}(A)$ is defined as
\[
	{\rm dom(A)} = \lk{ x \in L_2(\M) : \lim_{t \to 0_+} \frac{x - T_t(x)}{t} \text{ converges in } L_2(\M) }.
\]
Note that $A$ is a densely defined closed operator on $L_2(\M)$ with its domain ${\rm dom}(A)$. Moreover, $A$ is positive, and then $T_t = e^{-tA}$ on $L_2(\M)$. 

Let $\M^T = \{ x \in \M : T_t(x) = x \text{ for all } t \geq 0 \}$ be the fixed-point algebra of $(T_t)$. And define $E: \M \to \M^T$ as the conditional expectation. It is clear that
\[
\lim_{t \to \infty} T_t (x) = E(x), \quad \forall\, x \in \M.
\]    
We refer to \cite{Davies1992} for basic properties of quantum Markov semigroups. 

Now we introduce Dirichlet forms associated with quantum Markov semigroups. Indeed, $(T_t)$ admits a \emph{Dirichlet form} (cf. \cite[Theorem 2.7]{Albeverio1977})
\[
	\Eform: L_2(\M) \to [0, \infty],
\]
with the domain ${\rm dom}(A^{1/2}) \subseteq L_2(\M)$ (see \cite[Lemma 2.6]{Albeverio1977}) and
\[
	\Eform(x, x) = \lim_{t \to 0_+} \frac{ \tau\mk{ x^* ({\rm id} - T_t) (x) } }{t}, \quad \forall\, x \in {\rm dom}(A^{1/2}).
\]
We define $\Eform(x, y)$ via complex polarization for all $ x , y \in {\rm dom}(A^{1/2})$. In particular, if $x \in {\rm dom}(A)$,
\[
\Eform(x, x) = \tau ( x^* A(x) ).
\]

By definition (cf. \cite{Albeverio1977}), a Dirichlet form $\Eform$ is a quadratic form satisfying:
\begin{enumerate}[(1)]
	\item $\Eform$ is densely defined, positive and closed under its graph norm $\| x \|_\Eform^2 : = \| x \|_2^2+ \Eform(x, x)$;
	\item Let $f \in C(\real)$ be a Lipchitz function with $f(0) = 0$ and $\| f \|_{\rm Lip}$ denotes the Lipchitz norm, then
	\[
		x \in {\rm dom}(A^{1/2}) \text{ selfadjoint} \Rightarrow f(x) \in {\rm dom}(A^{1/2}) \text{ and } \Eform(f(x), f(x)) \leq \| f \|_{\rm Lip}^2 \Eform(x, x).
	\]
\end{enumerate}
\iffalse
A Dirichlet forms is completely if the amplification
\[
 \Eform^{(n)} \sk{ \sum_{i, j = 1}^n e_{i, j} \otimes x_{i,j}  } := \sum_{i, j = 1}^n \Eform( x_{i, j} ), \quad \text{ on } M_n(\com) \otimes {\rm dom}(A^{1/2})
\]
is a Dirichlet form for all $n \geq 1$.
\fi

The associated Dirichlet algebra $\Aalg_\Eform:= {\rm dom}(A^{1/2}) \cap \M$ constitutes a dense $\ast$-subalgebra of $\mathcal{M}$ that serves as a core for $A^{1/2}$. Moreover, we have $T_t(\Aalg_\Eform) \subseteq \Aalg_\Eform$. As shown in \cite[Lemma 7.2]{Cipriani2003}, $\Aalg_\Eform$ is closed under $C^1$-functional calculus. 

% For computational purposes in concrete examples, it is often advantageous to work with a regular subalgebra $\Aalg \subseteq \Aalg_\Eform$ enjoying stronger analytical properties. Throughout this paper, we abbreviate $\Aalg_\Eform$ to $\Aalg$ when no confusion arises. 

For $x, y \in {\rm dom}(A)$ with $x^*y \in {\rm dom}(A)$, recall the definition of Meyer's carr\'{e} du champ associated with the generator $A$
\[
	\Gamma(x, y) = \frac{1}{2} \sk{ (Ax)^* y + x^* Ay - A(x^* y) }.
\]
It is well-known that $\Gamma$ is a (completely) positive sesquilinear form. 
\iffalse
and the second order gradient form
\[
	\Gamma^2(x, y) = \frac{1}{2} \sk{ \Gamma(Ax, y) + \Gamma(x, Ay) - A\Gamma(x, y) }.
\]
A quantum Markov semigroup $(T_t)$ is said to be a \emph{nc diffusion} if  $$ \Gamma(x, x) \in L_1(\M), \quad \text{ for all } x \in {\rm dom}(A^{1/2}) \cap \M. $$
\fi

To avoid cumbersome domain considerations, we fix 
\[
\Aalg = \lk{ x \in \M : x,\ x^*x \in {\rm dom}(A) }.
\]
Therefore for every $x \in \Aalg$, $\Gamma(x, x)$ is well-defined in $L_1(\M)$. Moreover, for each $x, y$ such that $\Gamma(x, y)\in L_1(\M)$, one has
$$ \Eform(x, y) = \tau\sk{ \Gamma(x, y) }. $$

\subsection{A finite-time bimodule}
This subsection is devoted to the construction of a special bimodule, which is vital in the proof of Lemma \ref{new-lemma3.2}. We would like to remark that such special module has already been considered in \cite[Subsection 10.3]{Cipriani2003}. We first construct it at a fixed time $h > 0$ and then pass to the infinitesimal generator by letting $h\to 0$. 

Let
\[
A_h = \frac{{\rm id} - T_h}{h}, \quad \Eform_h(x, y) = \tau\sk{ x^* A_h(y) }, 
\]
and
\begin{align*}
      \Gamma_h(x, y) & = \frac{1}{2} \sk{ (A_hx)^* y + x^* A_h(y) - A_h(x^* y) } \\
      & = \frac{1}{2h} \sk{ x^*y + T_h(x^*y) - T_h(x^*)y - x^*T_h(y) }.
\end{align*}
%Both $\Eform$ and $\Gamma_h$ are positive sesquilinear forms. 
On the algebraic tensor $\M \odot \M$, define
\[
\lara{ w \otimes x, y \otimes z }_h = \frac{1}{2h} \tau \sk{ x^* T_h(w^* y) z }.
\]
For every
\[
\xi = \sum_{i = 1}^n x_i \otimes y_i \in \M \odot \M,
\]
since $T_h$ is completely positive, we have
\[
\lara{ \xi, \xi }_h = \frac{1}{2h} \sum_{i, j = 1}^n \tau\sk{ y_i^* T_h(x_i^*x_j) y_j } \geq 0.
\]
Hence $\lara{ \cdot, \cdot }_h$ is semi-positive definite. Faithfulness of $\tau$ yields
\[
\lara{ \xi, \xi }_h = 0 \iff \sum_{i, j = 1}^n y_i^* T_h(x_i^*x_j) y_j = 0.
\] 
Let
\[
\mathcal{Z}_h = \lk{ \sum_{i = 1}^n x_i \otimes y_i : x_i, y_i\in\M, \sum_{i, j = 1}^n y_i^* T_h(x_i^*x_j) y_j = 0,\ n \in \nat }.
\]
Define $\H_h$ as the completetion of $\M \odot \M/\mathcal{Z}_h$ under this inner product. Then $\H_h$ becomes a Hilbert space. For $a \in \M$, define 
\[
L_a ( x \otimes y ) = ax \otimes y, \quad R_a(x \otimes y) = x \otimes ya.
\]
Then $L_\cdot$ is a representation of $\M$ on $B(\H_h)$, while $R_\cdot$ is an anti-representation. 
We have the following proposition.

\begin{proposition}
For $x \in \M$, define the difference vector
\[
v_h(x) = x \otimes \1 - \1 \otimes x.
\]
Then
      \begin{enumerate}
            \item For any $x, y \in \M$,
            \begin{equation}\label{eq:v-Eform}
            \lara{ v_h(x), v_h(y) }_h = \Eform_h(x, y).
            \end{equation}
            \item For any $x \in \M$ and $a \in \M$,
            \begin{equation}\label{eq:v-LR}
            \lara{ v_h(x), L_a v_h(x) } = \tau\sk{ a \Gamma_h(x^*, x^*) }, \quad \lara{ v_h(x), R_a v_h(x) }_h = \tau\sk{ a \Gamma_h(x, x) }.
            \end{equation}
      \end{enumerate}
\end{proposition}

\begin{proof}
      Direct computation implies
      \begin{align*}
      2h\lara{ v_h(x), v_h(y) }_h & = \tau\sk{ T_h(x^*y) - x^*T_h(y) - T_h(x^*)y + x^*y } \\
      & = 2 \tau\sk{ x^*(I - T_h)(y) } = 2h \Eform_h(x, y).
      \end{align*}
      This proves \eqref{eq:v-Eform}. For the second identity, 
      \begin{align*}
      2h\lara{ v_h(x), R_a v_h(x) }_h & = 2h \lara{ x \otimes \1 - \1 \otimes x, x \otimes a - \1 \otimes xa }_h \\
      & = \tau\sk{ T_h(x^* x)a - x^*T_h(x)a - T_h(x^*)xa + x^*xa } \\
      & = 2h\tau \sk{ \Gamma_h(x, x)a }.
      \end{align*}
      The calulation for $L_a$ follows from the tracial property and symmetry of $T_h$, so we omit the details.
\end{proof}

\begin{remark}
This structure can be regarded as a finite-time version of the derivation bimodule structure introduced by Cipriani and Sauvageot \cite{Cipriani2003}: there exist a Hilbert $\M$-bimodule $\mathcal H$ and a closable
real
derivation
\[
 \delta:\Aalg  \to \mathcal H
\]
such that
\begin{equation*}
 \Eform(x,y)=\langle\delta x,\delta y\rangle_{\mathcal H}.
 \label{eq:gradient-realization}
\end{equation*}
For $x \in \M$, we define the left and right module actions $L_x$ and $R_x$ on $\mathcal H$ by
\[
L_x \xi = x\xi, \quad R_x \xi = \xi x,
\]
here left and right multiplication are understood in the sense of the bimodule structure. $L_x$ and $R_x$ are self-adjoint if $x$ is self-adjoint. Let $x \in \M_+$, if $(T_t)$ is a noncommutative diffusion, we have for self-adjoint $y\in\Aalg$,
\begin{equation*}
 \langle\delta y,L_x\delta y\rangle_{\mathcal H}
 =
 \langle\delta y,R_x\delta y\rangle_{\mathcal H}
 =\tau(x\Gamma(y,y)).
 \label{eq:left-right-gamma}
\end{equation*}
% For a general $y$, the second equality has $\Gamma(y^*,y^*)$ in place of $\Gamma(y,y)$.
The left and right representations commute, since
$L_xR_y\xi=x(\xi y)=(x\xi)y=R_yL_x\xi$; hence self-adjoint $L_x$ and $R_x$
admit a joint spectral calculus. However, one should assume the so-called ``$\Gamma$-regular'' condition, see, e.g. \cite{Wirth2021}, to ensure that $L_x$ and $R_x$ are normal on the whole $\M$. To avoid the discussion of the normality of $L_x$ and $R_x$, we utilize the current finite-time bimodule.
\end{remark}

It is easy to verify that $L_aR_b (x \otimes y) = ax \otimes yb = R_bL_a(x \otimes y)$, $L_a$ and $R_a$ are self-adjoint if $a$ is self-adjoint. Hence $L_a$ and $R_a$ are strongly commutative, they admit a joint spectral calculus for $a \in \M_{\rm sa}$. 

Let $I \subseteq \real$ be a closed interval and $\varphi \in C^1(I)$. Define
\[
\varphi^{[1]}(s,t)=
 \begin{cases}
 \dfrac{\varphi(s)-\varphi(t)}{s-t},&s\neq t,\\
 \varphi'(s),&s=t.
 \end{cases}
\]
We have the following chain rule. 

\begin{proposition}
    Let $x \in \M_{\rm sa}$ and $I$ be a compact interval containing the spectrum of $x$.  If $\varphi\in C^1(I)$, then $\varphi(x) \in \M $ and
\begin{equation}
 v_h(\varphi(x))
 =\varphi^{[1]}(L_x,R_x) v_h(x),
 \label{eq:doi-chain-rule}
\end{equation}
here $\varphi^{[1]}(L_x,R_x)$ is defined by the joint spectral calculus.  
\end{proposition}

\begin{proof}
We first prove for $\varphi(x) = x^n$, then
\[
\varphi^{[1]}(s, t) = \frac{ s^n - t^n }{ s - t } = \sum_{k = 0}^{n - 1} s^kt^{n - 1 -k}.
\]
Note that
\begin{align*}
v_h(xy) & = xy  \otimes \1 - \1 \otimes xy = (xy \otimes \1 - x \otimes y) + (x \otimes y - \1 \otimes xy) \\  
& = L_x(y \otimes \1 - \1 \otimes y ) + R_y (x \otimes \1 - \1 \otimes x) \\
& = L_x v_h(y) + R_y v_h(x).
\end{align*}
By induction, we deduce
\[
v_h(x^n) = \sum_{k = 0}^{n - 1} L_x^k R_x^{n - 1 - k} v_h(x) = \varphi^{[1]}(L_x, R_x)v_h(x).
\]
Therefore by linearity, this yields \eqref{eq:doi-chain-rule} for all polynomials $\varphi$. 

For $\varphi \in C^1(I)$, the Stone-Weierstrass theorem implies that there exist polynomials $P_\varepsilon$ such that
\[
\| P_\varepsilon - \varphi \|_\infty + \| P_\varepsilon' - \varphi' \|_\infty \to 0 \quad \text{as } \varepsilon \to 0.
\]
Then
\[
\sup_{s, t \in I} | P_\varepsilon^{[1]} - \varphi^{[1]} | \leq \| P_\varepsilon' - \varphi' \|_\infty \to 0.
\]
Then the joint spectral calculus yields 
\[
\| P_\varepsilon^{[1]}(L_x, R_x) - \varphi^{[1]}(L_x, R_x) \|_{B(\H_h)} \to 0.
\]
Note that
\[
\| v_h(x) \|_{\H_h}^2 = \lara{v_h(x), v_h(x)}_h = \frac{1}{h} \tau \sk{ x^* (I - T_h)(x) } \leq \frac{1}{h}\| x \|_2^2 \leq \frac{1}{h}\|x\|_\infty^2,
\]
we conclude that
\[
\| v_h( P_\varepsilon(x) ) - v_h(\varphi(x)) \|_{\H_h} \leq \frac{1}{\sqrt{h}}\| P_\varepsilon(x) - \varphi(x) \|_\infty \leq \frac{1}{\sqrt{h}}\| P_\varepsilon - \varphi \|_\infty.
\]
Therefore
\begin{align*}
& \| v_h(\varphi(x)) - \varphi^{[1]}(L_x, R_x) v_h(x) \|_{\H_h} \\
\leq\ & \| v_h( P_\varepsilon(x) ) - v_h(\varphi(x)) \|_{\H_h} 
+ \|  P_\varepsilon^{[1]}(L_x, R_x) v_h(x) - \varphi^{[1]}(L_x, R_x) v_h(x) \|_{\H_h} \\
\leq \ & \frac{1}{\sqrt{h}} \| P_\varepsilon - \varphi \|_\infty + \frac{1}{\sqrt{h}} \| P_\varepsilon^{[1]}(L_x, R_x) - \varphi^{[1]}(L_x, R_x) \|_{B(\H_h)} \| x \|_\infty 
\end{align*}
converges to $0$ as $\varepsilon \to 0$, this proves \eqref{eq:doi-chain-rule}.
\end{proof}

\bigskip

\section{$L_p$-Poincar\'{e} inequalities from Beckner inequalities}
\label{sec:heat-flow}
In this section, we aim to show that an $L_p$-Poincar\'{e} inequality can be deduced from a $q$-Beckner inequality for $q = p/(p-1)$. The proof combines Lemma \ref{new-lemma3.2} with integration along the heat flow. In Section~\ref{sec:PI-MLSI-to-p}, we will derive the corresponding constants of Beckner inequalities under the PI and MLSI assumptions.

Throughout this section, fix $1<q<2$, and let $\phi(t) = t^{q - 1}$. For $x \in \M_+$, define
\begin{equation*}
 D_q(x)=\tau(x^q)-\tau((E(x))^q).
 \label{eq:Dq-Iq}
\end{equation*}
The contractivity of $E$ on $L_q(\M)$ implies $D_q(x)\geq0$. For $x \in \M_+ \cap {\rm dom}(A)$, define 
\[
I_q(x) = \tau( A(x)x^{q - 1}).
\]
It is equal to $\Eform(x, x^{q-1})$ if $x^{q-1} \in {\rm dom}(A^{1/2})$. We say that $(T_t)$ satisfies a $q$-Beckner  inequality if there exists $\beta_q>0$ such that for all $x \in \M_+ \cap {\rm dom}(A)$,
\begin{equation}
 \beta_qD_q(x)\leq I_q(x).
 \label{eq:beta-q}
\end{equation}
The largest possible constant $\beta_q$ in \eqref{eq:beta-q} is called the \emph{Beckner constant}.

The following lemma is needed in the proof of Lemma \ref{new-lemma3.2}.
\begin{lemma}
For $s,t\geq0$ and $1<q<2$,
\begin{equation}
 \phi^{[1]}(s,t)^{-1}
 \leq\frac{s^{2-q}+t^{2-q}}{q-1},
 \label{eq:reciprocal-scalar}
\end{equation}
where $\phi^{[1]}(0,0)^{-1} = 0$ for consistence with the regularization
$(s,t) \mapsto (s+\varepsilon,t+\varepsilon)$.
\end{lemma}

\begin{proof}
If $s = t \geq 0$, \eqref{eq:reciprocal-scalar} is trivial. For $s > t \geq 0$, observe that $q - 2 < 0$ and
\begin{align*}
      \frac{s^{q-1}-t^{q-1}}{s-t}
& = (q-1) \int_0^1((1-u)t+us)^{q-2}\,\d u \\
& \geq \frac{q - 1}{ s^{2 - q} } \geq \frac{q - 1}{ s^{2 - q} + t^{2 - q} },
\end{align*}
which directly implies \eqref{eq:reciprocal-scalar}. The case $t > s \geq 0$ follows in the same way.
\end{proof}

\begin{lemma}\label{new-lemma3.2}
Let $1 < q < 2$. Let $x\in \M_+ \cap {\rm dom}(A)$ and $y \in \Aalg_{\rm sa}$. Then
\begin{equation}
 |\Eform(x,y)|^2\leq \frac{2}{q-1} I_q(x)\tau\bigl(x^{2-q}\Gamma(y,y)\bigr).
 \label{eq:weighted-cauchy}
\end{equation}
\end{lemma}

\begin{proof}
\iffalse
 For
$f_\varepsilon=f+\varepsilon\1$ define
\begin{equation}
 K_{q,\varepsilon}(f;y)
 =\langle\delta y,
  \theta_q(L_{f_\varepsilon},R_{f_\varepsilon})^{-1}\delta y
 \rangle_{\mathcal H},
 \qquad
 K_q(f;y)=\lim_{\varepsilon\downarrow0}K_{q,\varepsilon}(f;y).
 \label{eq:K-regularized}
\end{equation}
The limit exists because the quadratic forms in
\eqref{eq:K-regularized} decrease as $\varepsilon\downarrow0$.  
\fi

Let $h>0$ and define
\[
I_{q, h}(x) = \Eform_h(x, x^{q - 1}), \quad x \in \M_+.
\]
We first show that for every $h > 0$, $x\in \M_+ \cap {\rm dom}(A)$ and $y \in \Aalg_{\rm sa}$
\begin{equation}
 |\Eform_h(x,y)|^2\leq \frac{2}{q-1} I_{q, h}(x)\tau\bigl(x^{2-q}\Gamma_h(y,y)\bigr).
 \label{eq:weighted-cauchy-finite}
\end{equation}

Suppose first that $x \geq \varepsilon \1$ for some $\varepsilon>0$. Since the spectrum of
$x$ is contained in the compact interval $[\varepsilon,\|x\|_\infty + \varepsilon]$, the joint
spectrum $\sigma (L_x,R_x) \subseteq [\varepsilon,\|x\|_\infty + \varepsilon]^2$. Thus $ \phi^{[1]} $ is bounded above and bounded away from zero on $\sigma (L_x,R_x)$. Thus
$ \phi^{[1]}(L_x,R_x) $ is a positive bounded and invertible operator on
$\H_h$.  Hence
\begin{align*}
 |\Eform_h(x,y)| & = |\langle v_h(x), v_h(y) \rangle_{h}| = |\langle \phi^{[1]}(L_x,R_x)^{1/2} v_h(x),
        \phi^{[1]}(L_x,R_x)^{-1/2} v_h(y) \rangle_{h}|\\
 &\leq
 \langle v_h(x), \phi^{[1]}(L_x,R_x) v_h(x) \rangle_{h}^{1/2}
 \langle v_h(y), \phi^{[1]}(L_x,R_x)^{-1} v_h(y)\rangle_{h}^{1/2} \\
 & = I_{q, h}(x)^{1/2}
 \langle v_h(y), \phi^{[1]}(L_x,R_x)^{-1} v_h(y)\rangle_{h}^{1/2},
\end{align*}
where we use \eqref{eq:v-Eform} and the chain rule \eqref{eq:doi-chain-rule} in the last equality. Applying \eqref{eq:reciprocal-scalar} implies
\[
 \phi^{[1]}(L_x,R_x)^{-1}
 \leq\frac1{q-1}\bigl(L_{x^{2-q}}+R_{x^{2-q}}\bigr).
\]
Therefore by \eqref{eq:v-LR},
\begin{align*}
 \langle v_h(y), \phi^{[1]}(L_x,R_x)^{-1} v_h(y)\rangle_{h} 
 &\leq\frac1{q-1}
 \bigl(
 \langle v_h(y), L_{x^{2-q}} v_h(y)\rangle_h
 +\langle v_h(y), R_{x^{2-q}} v_h(y)\rangle_h
 \bigr)\\
 &=\frac2{q-1}\tau(x^{2-q}\Gamma_h(y,y)).
\end{align*}
This proves \eqref{eq:weighted-cauchy-finite} for $x \geq \varepsilon1>0$. 

For general $x \in \M_+ \cap {\rm dom}(A)$, apply the preceding argument to
$x_\varepsilon=x+\varepsilon \1$ for $\varepsilon>0$.  Since 
$v_h(x_\varepsilon) = v_h(x) + \varepsilon v_h(\1) = v_h(x)$, one has
\[
\Eform_h(x_\varepsilon,y) =\lara{ v_h(x_\varepsilon), v_h(y) }_h =\Eform_h(x,y), \quad I_{q, h}(x_\varepsilon) = \Eform_h(x_\varepsilon, x_\varepsilon^{q-1}) = \tau(A_h(x)x_\varepsilon^{q-1}).
\]
Then one has
\begin{equation*}
|\Eform_h(x,y)|^2\leq \frac{2}{q-1} \tau(A_h(x)x_\varepsilon^{q-1})\tau\bigl(x^{2-q}_\varepsilon \Gamma_h(y,y)\bigr).
\end{equation*}
Since $x_\varepsilon \to x$, letting $\varepsilon \to 0$ yields \eqref{eq:weighted-cauchy-finite}.

Now let $h \to 0$. Since $A_h(x) \to A(x)$ and $A_h(y) \to A(y)$ in $L_2(\M)$, we have
\[
\Eform_h(x, y) = \tau\sk{ x^* A_h(y) } \to \tau\sk{x^* A(y)} = \Eform(x, y),
\]
and
\[
I_{q, h}(x) = \tau( A_h(x) x^{q-1}) \to \tau( A(x)x^{q-1} ) = I_q(x).
\]
Note that $y \in \Aalg_{\rm sa}$, and we have $A_h(y^*y) \to A(y^*y)$ in $L_2(\M)$. Hence, one has
\begin{align*}
\tau(x^{2-q} \Gamma_h(y, y) ) & = \frac{1}{2} \mk{ \tau( x^{2-q}A_h(y^*)y ) + \tau( x^{2-q}y^* A_h(y) ) - x^{2-q}A_h(y^*y) } \\
& \to \frac{1}{2} \mk{ \tau( x^{2-q}A(y^*)y ) + \tau( x^{2-q}y^* A(y) ) - x^{2-q}A(y^*y) } \\
& = \tau(x^{2-q} \Gamma(y, y)). 
\end{align*}
Finally, by letting $h \to 0$ in \eqref{eq:weighted-cauchy-finite}, we have finished the proof of \eqref{eq:weighted-cauchy}.
\end{proof}

The following lemma is standard, but we still present the detailed proof for reader's convenience.

\begin{lemma}\label{lem:trace-power-dissipation}
Let $1<q<2$ and $x \in \M_+ \cap {\rm dom}(A)$. Denote by $x_t=T_t(x)$. Then $t\mapsto D_q(x_t)$ is continuous on $[0,\infty)$, locally absolutely
continuous on $(0,\infty)$, and for almost every $t>0$
\begin{equation}
 -\frac{\dd}{\dd t}D_q(x_t)=qI_q(x_t).
 \label{eq:single-flow-dissipation}
\end{equation}
Moreover,
\begin{equation}
 \lim_{t\to\infty}D_q(x_t)=0.
 \label{eq:single-flow-defect-limit}
\end{equation}
\end{lemma}

\begin{proof}
By definition,
\[
\tau( (E(x_t))^q ) = \tau( ( T_tE(x) )^q ) = \tau(E(x)^q).
\]
Hence
\[
\frac{\dd}{\dd t} D_q(x_t) = \frac{\dd}{\dd t} \sk{ \tau(x_t^q) - \tau( (E(x_t))^q ) } = \frac{\dd}{\dd t}\tau(x_t^q).
\]
By the spectral theorem, one has
\[
 \|AT_t\|_{L_2(\M)\to L_2(\M)}
 \leq\sup_{s\geq0}se^{-ts}<\infty,
 \quad t>0.
\]
Thus, $x_t\in {\rm dom}(A)$ and $\frac{\d}{\d t}x_t=-Ax_t$ in $L_2(\M)$ for $t > 0$, hence it also holds in $L_q(\M)$.
The trace differential implies
\begin{align*}
 \frac{\dd}{\dd t}\tau(x_t^q)
 =-q\tau(x_t^{q-1}A (x_t) )=-qI_q(x_t).
\end{align*}
\iffalse
The second equality is precisely
Lemma~\ref{lem:generator-singular-pairing}, applied to
$f_t\in\Dom(A)\cap\M$; traciality permits the two factors in its scalar
pairing to be interchanged.  It does not assert that $f_t^{q-1}$ is a form
element.  Since $Ef_t=ET_tf=Ef$, the
term $\tau((Ef_t)^q)$ is constant.  This proves
\eqref{eq:single-flow-dissipation}; the same computation on compact
positive-time intervals gives local absolute continuity.
\fi
Finally, when $t\to\infty$, $T_t(x)\to E(x)$ in $L_2(\M)$ and this convergence also holds in $L_q(\M)$. This yields
$\tau(x_t^q)\to\tau((E(x))^q)$, which proves
\eqref{eq:single-flow-defect-limit}.  
% Strong continuity at zero in $L_2$, and hence in $L_q$, proves the asserted continuity at zero.
\end{proof}

We employ the $q$-Beckner inequality \eqref{eq:beta-q} to show the following $L_p$-Poincar\'{e} inequality. In the meantime, we also derive the associated constants.

\begin{theorem}
\label{thm:single-defect}
Let $p>2$ and $q = p/(p - 1)$.
Assume that \eqref{eq:beta-q} holds with $\beta_q>0$.  Then every $x \in \Aalg_{\rm sa}$ satisfies
\begin{equation}
 \|x-E(x)\|_p
 \leq
 \frac{2^{3/2+1/p}}
 {q\sqrt{\beta_q(q-1)}}
 \|\Gamma(x,x)^{1/2}\|_p.
 \label{eq:master-conclusion}
\end{equation}
\end{theorem}

\begin{proof}
Let $y=x-E(x)$.  Then $y \in \Aalg_{\rm sa}$  and $E(y)=0$. Moreover, $\Gamma(y,y)=\Gamma(x,x)$. Denote by
\begin{equation}
 M=\tau(|y|^p),
 \qquad
 H=\|\Gamma(y,y)^{1/2}\|_p.
 \label{eq:M-H}
\end{equation}
It suffices to consider $M > 0$ and $H < \infty$. We devide the proof into three steps.

\medskip
\noindent\emph{\textbf{Step 1}:}
Define
\iffalse
\begin{equation*}
 a_+=y_+^{p-1},
 \qquad
 a_-=y_-^{p-1},
 \qquad
 J_p(y)=a_+-a_-=\sgn(y)|y|^{p-1},
 \label{eq:dual-seed}
\end{equation*}
\fi
\begin{equation*}
 f_t^+=T_t (y_+^{p-1}),
 \quad
 f_t^-=T_t(y_-^{p-1}).
 \label{eq:positive-flows}
\end{equation*}
Then $f_t^{\pm} \in \M_+ \cap {\rm dom}(A)$ for $t > 0$. Define
\begin{equation*}
 D_t^\pm=D_q(f_t^\pm),
 \quad
 I_t^\pm=I_q(f_t^\pm),
 \quad
 D_t=D_t^++D_t^-,
 \quad I_t=I_t^++I_t^-.
 \label{eq:Dt-It}
\end{equation*}
By Lemma~\ref{lem:trace-power-dissipation}, for almost every $t\in (0,\infty)$, one has
\begin{equation}
 - \frac{\d}{\d t} D_t^\pm=qI_t^\pm,
 \quad
 - \frac{\d}{\d t} D_t=qI_t,
 \label{eq:D-derivative}
\end{equation}
and
\begin{equation}
 \lim_{t\to\infty}D_t^\pm=0,
 \quad
 \lim_{t\to\infty}D_t=0.
 \label{eq:D-infinity}
\end{equation}
By the $q$-Beckner inequality \eqref{eq:beta-q},
\begin{equation}
 I_t\geq\beta_qD_t.
 \label{eq:I-lower-D}
\end{equation}
From the definition of $D_t$, it is clear that $D_t$ is nonincreasing.  Fix $0<\delta<R$ in an interval on which $D_t>0$. Then by \eqref{eq:D-derivative} and \eqref{eq:I-lower-D}, 
\begin{align*}
 \int_\delta^R\sqrt{I_t}\,\dd t
 &=\frac1q\int_\delta^R\frac{-D_t'}{\sqrt{I_t}}\,\dd t \leq\frac1{q\sqrt{\beta_q}}
 \int_\delta^R\frac{-D_t'}{\sqrt{D_t}}\,\dd t\\
 &=\frac{2}{q\sqrt{\beta_q}}
   \bigl(\sqrt{D_\delta}-\sqrt{D_R}\bigr).
\end{align*}
If $D_t$ first vanishes at time $t_0$, then it is zero for all
$t\geq t_0$, and \eqref{eq:D-derivative} implies $I_t=0$ for almost every such
$t$. Thus the preceding estimate may be integrated up to $t_0$.  If no zero
occurs, let $R\to\infty$ and use \eqref{eq:D-infinity}.  Finally let
$\delta\downarrow0$, use continuity of $D_t$ at zero, and apply monotone
convergence to the nonnegative integral.  In both cases,
\begin{equation*}
 \int_0^\infty\sqrt{I_t}\,\dd t
 \leq\frac{2}{q\sqrt{\beta_q}}\sqrt{D_0}.
 \label{eq:length-before-D0}
\end{equation*}
Since $E(y_\pm^{p-1})\geq0$, one has
\[
D_q(y_\pm^{p-1}) = \tau( y_\pm^{(p-1)q} ) - \tau\sk{ (E(y_\pm^{p-1}))^q } \leq \tau( y_\pm^{p} ) .
\] 
By Lemma~\ref{lem:dual-seed},
\begin{align*}
 D_0 \leq \tau(y_+^p)+\tau(y_-^p) = M,
\end{align*}
and consequently
\begin{equation}
 \int_0^\infty\sqrt{I_t}\,\dd t
 \leq\frac{2}{q\sqrt{\beta_q}}M^{1/2}.
 \label{eq:length-final}
\end{equation}

\smallskip
\noindent\emph{\textbf{Step 2}:}
Define
\[
 K_t^\pm= \frac{2}{q-1}\tau\sk{ (f_t^\pm)^{2-q} \Gamma(y, y) },
 \quad K_t=K_t^++K_t^-.
\]
By \eqref{eq:weighted-cauchy} and the Cauchy--Schwarz inequality, one has
\begin{align}
 |\Eform(f_t^+-f_t^-,y)|
 \leq
 \sqrt{I_t^+K_t^+}+\sqrt{I_t^-K_t^-}
  \leq\sqrt{I_tK_t}.
 \label{eq:two-sign-Cauchy}
\end{align}
By the H\"older inequality,
\begin{equation}
 \tau((f_t^\pm)^{2-q}\Gamma(y,y)) \leq
 \|(f_t^\pm)^{2-q}\|_{p/(p-2)}
 \|\Gamma(y,y)\|_{p/2} =\|f_t^\pm\|_q^{2-q} H^2 \label{eq:K_t-1}.
\end{equation}
Since $T_t$ is contractive on $L_q(\M)$, we have
\[
 \|f_t^\pm\|_q^{2-q}
 \leq\| y_\pm^{p-1} \|_q^{2-q} = \tau(y_\pm^p)^{(2 - q)/q}
 =\tau(y_\pm^p)^{(p-2)/p}.
\]
Note that for $0< (p-2)/p <1$, $u\geq 0$ and $v\geq 0$
\[
 u^{(p-2)/p}+v^{(p-2)/p} \leq 2^{2/p} (u+v)^{(p-2)/p} .
\]
Substituting into \eqref{eq:K_t-1} yields
\begin{equation}
K_t
 \leq \frac{2^{1+2/p}}{q-1} \tau( y_+^p + y_-^p )^{(p-2)/p} H^2 =
 \frac{2^{1+2/p}}{q-1}
 M^{(p-2)/p}H^2.
 \label{eq:Kt-final}
\end{equation}

\smallskip
\noindent\emph{\textbf{Step 3}:} Recall $J_p(y) = \sgn(y)|y|^{p-1}$. By Lemma~\ref{lem:dual-seed},
$\tau(J_p(y)y)=\tau(|y|^p)=M$.  Moreover, since $E(y) = 0$, 
\begin{equation}
 \tau(E(J_p(y)) y)=\tau \mk{E(J_p(y)) E(y) }=0.
 \label{eq:fixed-dual-pairing-zero}
\end{equation}
For $0<\varepsilon<R$,
\begin{align*}
 \tau((T_\varepsilon J_p(y)-T_RJ_p(y))y)
 &=\int_\varepsilon^R\tau(AT_tJ_p(y)y)\,\dd t\\
 &=\int_\varepsilon^R\Eform(T_tJ_p(y),y)\,\dd t.
\end{align*}
\iffalse
Indeed, $J_p(y)\in L_2$, and the spectral bound
$\|AT_t\|_{2\to2}\leq\sup_{r\geq0}re^{-tr}$ places
$T_tJ_p(y)$ in $\Dom(A)$ for positive time.  The identity therefore first
holds as an $L_2$ Bochner integral and then after pairing with the bounded core
element $y$.  
\fi
As $R\to\infty$, $T_RJ_p(y)\to E(J_p(y))$ in $L_2(\M)$, and also in $L_q(\M)$.  Observe that $T_\varepsilon J_p(y)\to J_p(y)$ in $L_q(\M)$ as
$\varepsilon\downarrow0$.  
% Noncommutative H\"older makes both endpoint pairings continuous because $1/p+1/q=1$.  
Together with \eqref{eq:dual-seed-pairing} and
\eqref{eq:fixed-dual-pairing-zero}, the left side equals
\[
\lim_{\varepsilon\downarrow0, R\to\infty} \tau((T_\varepsilon J_p(y)-T_RJ_p(y))y) = \tau(J_p(y)y) - \tau(E (J_p(y) )y) = M.
\]
Since $T_tJ_p(y)=f_t^+-f_t^-$, one has
\begin{equation}
 M=\int_0^\infty\Eform(f_t^+-f_t^-,y)\,\dd t,
 \label{eq:heat-reconstruction}
\end{equation}
where the integral is absolutely convergent. Indeed, it follows from
\eqref{eq:two-sign-Cauchy}, \eqref{eq:length-final}, and
\eqref{eq:Kt-final}. Finally, combining \eqref{eq:heat-reconstruction},
\eqref{eq:two-sign-Cauchy}, \eqref{eq:length-final}, and
\eqref{eq:Kt-final}, we obtain
\begin{align*}
 M
 \leq \sk{ \sup_{t > 0} K_t }^{1/2}
       \int_0^\infty\sqrt{I_t}\,\dd t \leq
 \frac{2^{3/2+1/p}}
 {q\sqrt{\beta_q(q-1)}}
 M^{(p-1)/p}H.
\end{align*}
This completes the proof of \eqref{eq:master-conclusion}.
\end{proof}

By a standard argument, Theorem~\ref{thm:single-defect} extends to general non-selfadjoint elements. We restate this result as follows.
\begin{theorem}\label{thm:non-selfadjoint}
Let $p>2$ and $q = p/(p - 1)$.
Assume that \eqref{eq:beta-q} holds with $\beta_q>0$.  Then every $x \in \Aalg$ with $x^* \in \Aalg$ satisfies
\[
 \|x-E(x)\|_p
 \leq
 \frac{2^{5/2-1/p}}
 {q\sqrt{\beta_q(q-1)}}
 \max\lk{ \|\Gamma(x,x)^{1/2}\|_p, \| \Gamma(x^*, x^*)^{1/2}\|_p }.
 \]
\end{theorem}

\begin{proof}
Let $y = x - E(x)$ and
\[ y_1 = \frac{y + y^*}{2}, \quad y_2 = \frac{y - y^*}{2i}. \]
Then $y_i \in \Aalg_{\rm sa}$ and $E(y_i) = 0$, $i = 1,2$. By the Clarkson inequality,
\[
\| y_1 + iy_2 \|_p^p + \| y_1 - iy_2 \|_p^p \leq 2^{p - 1}(\| y_1 \|_p^p + \| y_2 \|_p^p),
\]
which implies
\begin{equation}\label{eq:y-1-2}
      \| y \|_p^p \leq 2^{p - 2}(\| y_1 \|_p^p + \| y_2 \|_p^p).
\end{equation}
On the other hand, since $\Gamma$ is sesquilinear,
\[
\Gamma(y_1, y_1) + \Gamma(y_2, y_2) = \frac{1}{2} \sk{ \Gamma(y, y) + \Gamma(y^*, y^*) }.
\]
Therefore by the McCauchy trace inequality and the convexity inequality \cite[Lemma~2.5-2.6]{McCauchy67}, we obtain
\begin{align*}
\tau\sk{ \Gamma(y_1, y_1)^{p/2} } + \tau\sk{ \Gamma(y_2, y_2)^{p/2} }
\leq & \tau\mk{ \sk{ \Gamma(y_1, y_1) + \Gamma(y_2, y_2) }^{p/2} } \\
= & \tau\mk{ \sk{ \frac{\Gamma(y, y) + \Gamma(y^*, y^*)}{2} }^{p/2} } \\
\leq & \frac{1}{2} \tau\sk{ \Gamma(y, y)^{p/2} + \Gamma(y^*, y^*)^{p/2} } \\
\leq & \max\lk{ \|\Gamma(x,x)^{1/2}\|_p, \| \Gamma(x^*, x^*)^{1/2} \|_p }^p.
\end{align*}
By applying Theorem \ref{thm:single-defect} to $y_1$ and $y_2$ seperately, and substituting it into \eqref{eq:y-1-2}, we obtain the desired result. 
\end{proof}

\bigskip

\section{$L_p$-Poincar\'{e} inequalities from PI and MLSI---Tracial case}\label{sec:PI-MLSI-to-p}
To show Theorem \ref{Main-p-PI}, it remains to derive Beckner constants of Beckner inequalities from PI and MLSI by Theorem \ref{thm:single-defect}. For the proof of $L_p$-Poincar\'{e} inequalities from PI, the following noncommutative Stroock-Varopouls inequality will be used, and the detailed proof can be found in \cite{Xiong2017}.
\begin{lemma}
\label{lem:SV}
For $1<q<2$ and every $x\in \M_+\cap {\rm dom}(A)$,
\begin{equation}
 \Eform(x^{q/2},x^{q/2})
 \leq\frac{q^2}{4(q-1)}I_q(x).
 \label{eq:SV}
\end{equation}
\end{lemma}

\subsection{$L_p$-Poincar\'{e} inequalities from PI}
We first combine \eqref{eq:spectral-gap} with the above noncommutative Stroock-Varopouls inequality.

\begin{proposition}
\label{prop:gap-to-q}
Assume that \eqref{eq:spectral-gap} holds. Let $\lambda$ be the spectral gap. Then for every
$1<q<2$ and every
$x\in \M_+ \cap {\rm dom}(A)$,
\begin{equation}
 \frac{4(q-1)\lambda}{q^2}D_q(x)\leq I_q(x).
 \label{eq:gap-q-defect}
\end{equation}
Moreover, the constant $4(q-1)\lambda/q^2$ is best possible.
\end{proposition}

\begin{proof}
Let $g=x^{q/2}$.  Since $q/2\in(0,1)$, the function
$u\mapsto u^{q/2}$ is operator concave.  The operator Jensen inequality for the unital
completely positive map $E$ yields
\begin{equation*}
E(g) = E(x^{q/2})\leq(Ex)^{q/2}.
 \label{eq:concave-Jensen}
\end{equation*}
\iffalse
If $0\leq a\leq b$, then
\[
 \tau(b^2-a^2)=\tau((b-a)(b+a))
 =\tau((b-a)^{1/2}(b+a)(b-a)^{1/2})\geq0.
\]
Applying this to \eqref{eq:concave-Jensen} yields
\begin{equation*}
 \tau((Eg)^2)\leq\tau((Ef)^q).
 \label{eq:squared-Jensen-trace}
\end{equation*}
\fi
Consequently,
\begin{align*}
	\tau((Ex)^q)=\| (Ex)^{q/2}\|_2^2\geq \| E(x^{q/2})\|_2^2
\end{align*}
and
\begin{align*}
 D_q(x)
=\tau(g^2)-\tau((Ex)^q)
\leq\tau(g^2)-\tau((Eg)^2)
 =\|g-Eg\|_2^2.
 \label{eq:Dq-by-L2-variance}
\end{align*}
\iffalse
The final equality follows because $E$ is the $L_2$-orthogonal projection:
\begin{align*}
 \|g-Eg\|_2^2
 &=\tau(g^2)-2\tau(gEg)+\tau((Eg)^2)\\
 &=\tau(g^2)-\tau((Eg)^2),
\end{align*}
where $\tau(gEg)=\tau(E(gEg))=\tau((Eg)^2)$.
\fi
Applying \eqref{eq:spectral-gap} to $g$ and Lemma~\ref{lem:SV}, one has
\begin{align*}
 \lambda D_q(x)
 \leq\lambda\|g-Eg\|_2^2\leq\Eform(g,g)\leq\frac{q^2}{4(q-1)}I_q(x),
\end{align*}
which is exactly \eqref{eq:gap-q-defect}. 
% The form-domain membership of $g=f^{q/2}$, including the case in which $f$ has a kernel, was proved in Lemma~\ref{lem:SV} by closedness.

Next we show that the constant $4(q-1)\lambda/q^2$ is best possible, namely, there exists a special semigroup attaining this constant. Consider the diffusion semigroup on $\real_+ = [0, \infty)$:
\[
\mathscr L (f) = f'' - 2f', \quad T_t(f) = e^{t\mathscr L}.
\]
The invariant measure and Dirichlet form are
\[
\d \mu(t) = 2e^{-2t} \,\d t, \quad \Eform(f, g) = \int_{\real_+} f'(t) g'(t) \,\d \mu(t).
\]
Then $(T_t)$ is a quantum Markov semigroup on $L_\infty(\real_+, \d \mu)$ with the fixed-point subalgebra $\com$. One can check that the spectral gap $\lambda = 1$.

For $0 < a < 2/q$, consider the function $f_a(t) = e^{at}$. Then
\[
\mu(f_a^q)  = \int_{\real_+} f_a(t)^q \,\d \mu(t) = \frac{2}{2 - qa}, \quad \mu(f_a) = \frac{2}{2 - a}.
\]
Hence
\[
D_q(f_a) = \mu(f_a^q) - \mu(f_a)^q = \frac{2}{2 - qa} - \sk{ \frac{2}{2 - a} }^q,
\]
and
\[
I_q(f_a) = \Eform(f_a, f_a^{q-1}) = (q-1) \frac{2a^2}{2 - qa}.
\]
Therefore, the beckner constant
\[
\beta_q \leq \frac{ I_q(f_a) }{ D_q(f_a) } = (q-1)a^2 \frac{ 1 }{ 1 - \frac{2 - qa}{2}\sk{ \frac{2}{2-a} }^q }.
\]
By letting $a \to 2/q$, we obtain
\[
\beta_q \leq \frac{4(q - 1)}{q^2}.
\]
However \eqref{eq:gap-q-defect} implies $ \beta_q \geq 4(q-1)/q^2 $, which proves the sharpness.
\end{proof}

\begin{proof}[Proof of \eqref{main-1} in Theorem \ref{Main-p-PI}]
For $p > 2$ and $x \in \Aalg_{\rm sa}$, let $q = p/(p-1)$. By Proposition \ref{prop:gap-to-q}, we can choose $\beta_q=4(q-1)\lambda/q^2$.  Hence, by Theorem \ref{thm:single-defect}, we obtain
   \[
   \| x - E(x) \|_p \leq \frac{ 2^{3/2 + 1/p}  }{ q\sqrt{ 4(q - 1)^2\lambda/q^2 } } \| \Gamma(x, x)^{1/2} \|_p = \frac{ 2^{1/2 + 1/p}(p - 1) }{\sqrt{\lambda}} \| \Gamma(x, x)^{1/2} \|_p.
   \]
   Since $2^{1/2 + 1/p}(p - 1) \leq \sqrt{2} p$ for $p \geq 2$, we finish the proof.
\end{proof}

\subsection{$L_p$-Poincar\'{e} inequalities from MLSI}

The key idea is the point that $D_q(x)$ can be written as an integral representation of relative entropy. This integral representation relies on the following elementay fact.  
% It avoids both a diffusion chain rule and any non-self-adjoint $\M$-valued expression such as $\Gamma(f^{q-1},f)$.

\begin{lemma}
\label{lem:shifted-entropy}
Let $1<q<2$. For $s>0$ and $u\geq0$, define
\begin{equation*}
 \Psi_s(u)
 =(u+s)\log(u+s)-s\log s-u(1+\log s).
 \label{eq:Psi}
\end{equation*}
Then
\begin{equation}
 u^q=c_q\int_0^\infty\Psi_s(u)s^{q-2}\,\dd s,
 \label{eq:power-entropy-rep}
\end{equation}
\begin{equation}
 qu^{q-1}
 =c_q\int_0^\infty
 \bigl(\log(u+s)-\log s\bigr)s^{q-2}\,\dd s,
 \label{eq:power-log-rep}
\end{equation}
where
\begin{equation*}
 c_q=\frac{q(q-1)\sin(\pi(q-1))}{\pi}.
 \label{eq:cq}
\end{equation*}
Both integrals converge absolutely for $u$ in a bounded subset of
$[0,\infty)$.
\end{lemma}

\begin{proof}
First note that for each fixed $u$, 
\[
|\Psi_s(u)| + |\log(u+s)-\log s| \leq C_u \min \lk{ 1 + |\log s|, s^{-1} } .
\]
Since $1 < q < 2$, the two integrals are absolutely convergent.

We have $\Psi_s(0)=\Psi_s'(0)=0$ and
\[
 \Psi_s''(u)=\frac1{u+s}.
\]
Thus $\Psi_s(u) \geq \Psi_s(0) = 0$. For $u > 0$, we observe
\[
\frac{\d^2}{ \d u^2}\int_0^\infty \Psi_s(u) s^{q-2} \,\d s = \int_0^\infty\frac{s^{q-2}}{u+s}\,\d s = u^{q-2}\int_0^\infty\frac{v^{q-2}}{1+v}\,\d v,
\]
where we substitute $s=uv$ in the last equality. 
\iffalse
and the beta integral give
\begin{align*}
 \int_0^\infty\frac{s^{q-2}}{u+s}\,\dd s
 &=u^{q-2}\int_0^\infty\frac{v^{q-2}}{1+v}\,\dd v\\
 &=\frac{\pi}{\sin(\pi(q-1))}u^{q-2}.
\end{align*}
\fi
Solving this ODE yields
\[
\int_0^\infty \Psi_s(u) s^{q-2} \,\d s = \frac{1}{q(q-1)}u^q\int_0^\infty\frac{v^{q-2}}{1+v}\,\d v = \frac{1}{c_q} u^q.
\]
This proves \eqref{eq:power-entropy-rep}.
Differentiating it in $u$ yields \eqref{eq:power-log-rep}.  
\iffalse
This differentiation under the
integral is justified first on compact $u$-intervals by the bounds
$|\log(u+s)-\log s|=O(|\log s|+1)$ as $s\downarrow0$ and $O(s^{-1})$ as
$s\to\infty$, with constants uniform on the compact interval.  Absolute
convergence of the first identity follows from
\[
 \Psi_s(u)=O(|\log s|+1)\quad(s\downarrow0),
 \qquad
 \Psi_s(u)=O(s^{-1})\quad(s\to\infty),
\]
uniformly for bounded $u$, together with $1<q<2$.
\fi
\end{proof}

\begin{theorem}
\label{prop:mlsi-to-q}
Assume that \eqref{eq:mlsi} holds with constant $\alpha$. Then, for every $1<q<2$ and every $x \in \M_+ \cap {\rm dom}(A)$,
\begin{equation}
 \frac{\alpha}{q} D_q(x)\leq I_q(x).
 \label{eq:mlsi-q-defect}
\end{equation}
Moreover, the constant $\alpha/q$ is best possible.
\end{theorem}

\begin{proof}
For $s > 0$ and $x \in \M_+ \cap {\rm dom}(A)$, one has
\begin{align*}
\Psi_s(x) - \Psi_s( E(x) ) & = (x + s\1)\log (x + s\1) - ( E(x) + s\1)\log ( E(x) + s\1) \\
& \qquad - (x - E(x))(1 + \log s).
\end{align*} 
By applying \eqref{eq:power-entropy-rep} and taking traces, we obtain
\begin{align*}
 D_q(x)
 &=c_q\int_0^\infty
 \Bigl[
 \tau((x+s\1)\log(x+s\1))
 -\tau((E(x)+s\1)\log(E(x)+s\1))
 \Bigr]s^{q-2}\,\dd s \notag\\
 &=c_q\int_0^\infty\Ent_E(x+s\1)s^{q-2}\,\dd s,
 \label{eq:Dq-entropy-integral}
\end{align*}
where we use the fact that
$E(x+s\1)=E(x)+s\1$ and $\tau(E(x))=\tau(x)$.
For each $s>0$, $x+s\1$ is invertible.  Then by \eqref{eq:mlsi}, one has
\begin{align*}
 \alpha\Ent_E(x+s\1)
 \leq\Eform(x+s\1,\log(x+s\1))
 =\Eform(x,\log(x+s\1)).
\end{align*}
This implies
\begin{align*}
 \alpha D_q(x) & \leq c_q\int_0^\infty
 \Eform(x,\log(x+s\1))s^{q-2}\,\dd s\\
 & =
 c_q\int_0^\infty
 \Eform(x,\log(x+s\1)-\log s\,\1)s^{q-2}\,\dd s\\
 & =qI_q(x),
\end{align*}
where we use $\Eform(x,\log s\,\1) = 0$ and \eqref{eq:power-log-rep}. This proves \eqref{eq:mlsi-q-defect}.
\iffalse
For completeness, no cancellation of divergent operator integrals is used.
For $0<\varepsilon<R<\infty$, define
\[
 G_{\varepsilon,R}(u)
 =c_q\int_\varepsilon^R
   (\log(u+s)-\log s)s^{q-2}\,\dd s.
\]
The entropy integrands are nonnegative by operator Jensen.  On the energy
side, the divided-difference kernel of $u\mapsto\log(u+s)$ is positive, so
\begin{align*}
 &c_q\int_\varepsilon^R
 \Eform(f,\log(f+s\1)-\log s\,\1)s^{q-2}\,\dd s\\
 &\qquad
 =\langle\delta f,G_{\varepsilon,R}^{[1]}(L_f,R_f)\delta f\rangle_{\mathcal H}
 \geq0.
\end{align*}
As the interval $[\varepsilon,R]$ increases to $(0,\infty)$, these positive
joint-spectral kernels increase.  Tonelli's theorem and monotone convergence
for positive quadratic forms therefore apply on both sides.  Formula
\eqref{eq:power-log-rep}, first for scalar spectral values and then in the
joint functional calculus, identifies the limiting positive quadratic form
with $qI_q(f)$.

If $x$ is not invertible, the argument can still be applied for $x_\varepsilon = x + \varepsilon\1$, and then we have
\[
\alpha D_q(x_\varepsilon) \leq q I_q(x_\varepsilon).
\]
By letting $\varepsilon \to 0$, the proof is completed.
\fi
\medskip 

Now we prove the sharpness of the constant. Consider $\{ -1, 1 \}$ equipped with the normalized counting measure. For any $f : \{ -1, 1 \} \to \com$, the expectation is $E(f) = (f(-1) + f(1))/2$. Take
\begin{equation*}
 A= I-E, \quad T_t = e^{-tA} .
 \label{eq:balanced-two-point-generator}
\end{equation*}
Then $(T_t)$ is a quantum Markov semigroup on $L_\infty(\{-1, 1\})$ with the fixed-point subspace $\com \1$. We first verify the optimal MLSI constant. Define the coordinate function
\[
\omega(x) = x, \quad x \in \{ -1, 1 \}.
\]
Then $\{1, \omega\}$ is a linear basis of $L_\infty(\{-1, 1\})$. For every positive $f \in L_\infty(\{-1, 1\})$, one can write
\[
f = C (1 + s \omega), \quad C > 0,\ -1 < s < 1.
\]
Without loss of generality, let $C= 1$ and $0 \leq s < 1$. Then
\[
\Ent_E(f) = \frac{1}{2} \sk{ (1 + s)\log(1 + s) + (1 - s)\log(1 - s) }  ,
\]
and
\[
\Eform(f, \log f) = \frac{1}{2} \sk{ (1 + s)\log(1 + s) + (1 - s)\log(1 - s) - ( \log(1+s) + \log(1-s) ) }.
\]
Let $F(s) = \Eform(f, \log f) - 2\Ent_E(f)$, then $F(0) = F'(0) = 0$ and
\[
F''(s) = -\frac{1}{2} \sk{ \frac{1}{s + 1} - \frac{1}{s - 1} - \frac{1}{(s + 1)^2} - \frac{1}{(s - 1)^2} } = \frac{2s^2}{(s^2 - 1)^2} \geq 0.
\]
Hence $F(s) \geq 0$ and therefore
\[
2\Ent_E(f) \leq \Eform(f, \log f), \quad \forall\, f \text{ positive}.
\]
Conversly, for $f_\varepsilon = 1 + \varepsilon \omega$, we have
\[
\Ent_E(f_\varepsilon) = \frac{\varepsilon^2}{2} + o(\varepsilon^2), \quad \Eform(f_\varepsilon, \log f_\varepsilon) = \varepsilon^2 + o(\varepsilon^2).
\]
We conclude that the optimal MLSI constant is $\alpha = 2$. Theorem \ref{prop:mlsi-to-q} implies the Beckner constant $\beta_q \geq 2/q$. Moreover, direct computation yields
\[
D_q(f_\varepsilon) = \frac{q(q-1)}{2}\varepsilon^2 + o(\varepsilon^2), \quad I_q(f_\varepsilon) = (q-1)\varepsilon^2 + o(\varepsilon^2).
\]
Then $\beta_q \leq 2/q$. This proves the sharpness of the constant $\alpha/q$.
\end{proof}

\begin{remark}
Theorem \ref{prop:mlsi-to-q} also answers the problem on the best universal comparison between the MLSI constant and the Beckner constant posed in \cite[Remark~2.10]{Adamczak2022} in the
scalar-fixed-point reversible setting.
\end{remark}

\begin{proof}[Proof of \eqref{main-2} in Theorem \ref{Main-p-PI}]
For $p>2$ and $x \in \Aalg_{\rm sa}$, let $q = p/(p-1)$. By Theorem \ref{prop:mlsi-to-q}, we can choose $\beta_q=\alpha/q$.  Hence, by Theorem \ref{thm:single-defect}, we obtain
\begin{align*}
 \|x-Ex\|_p
\leq
 \frac{2^{3/2+1/p}}
 {q\sqrt{(\alpha/q)(q-1)}}
 \|\Gamma(x,x)^{1/2}\|_p =\frac{2^{3/2+1/p}(p-1)}
 {\sqrt{\alpha p}}
 \|\Gamma(x,x)^{1/2}\|_p.
\end{align*}
Since $2^{3/2+1/p}(p-1)\leq 2\sqrt{2}p$, we finish the proof.
\end{proof}

The $L_p$-Poincar\'{e} inequality with $O(\sqrt{p})$ coefficient implies the following deviation inequality and exponential integrability (cf. \cite{Efraim2008, Bobkov1999, Junge2015}). 
\begin{corollary}
	Suppose that \eqref{eq:mlsi} holds with constant $\alpha > 0$. Then for any $x \in \Aalg_{\rm sa}$,
	\[
		\tau\sk{ e^{|x - E(x)|} } \leq 2\exp\sk{ 16\alpha^{-1} \| \Gamma(x, x) \|_{\infty} }, 
	\]
	and
	\[
		\tau\sk{ \1_{ [t, \infty) } (|x - E(x)|) } \leq 2\exp\sk{ -\frac{\alpha t^2}{ 64 \| \Gamma(x, x) \|_\infty } }, \quad \forall\, t > 0.
	\] 
\end{corollary} 

\begin{proof}
	It suffices to consider selfadjoint $x \in \Aalg$ with $E(x) = 0$. Obverve that
	\[
		\tau( e^{|x|} ) \leq \tau(e^{x}) + \tau(e^{-x}) = 2\tau\sk{ \cosh (x) }. 
	\] 
	Invoking Taylor's expansion and the moment estimate \eqref{main-2} with $p= 2k$, we have
	\[
	\begin{aligned}
		\tau( \cosh (x) ) & = 1 + \sum_{k = 1}^{\infty} \frac{ \tau(x^{2k}) }{ (2k)! } \leq 1 + \sum_{k = 1}^{\infty} \frac{ 16^{k}k^k }{ \alpha^k (2k)! } \| \Gamma(x, x) \|_k^k \\
		& \leq 1 + \sum_{k = 1}^{\infty} \frac{ 16^{k} }{ \alpha^k k! } \| \Gamma(x, x) \|_\infty^k = \exp\sk{ 16\alpha^{-1}\| \Gamma(x, x) \|_\infty }.
	\end{aligned}
	\]
	The tail bound follows from optimizing $\lambda > 0$ in
	\[
		\tau\sk{ \1_{ [t, \infty) } (|x|) } \leq e^{-\lambda t} \tau\sk{ e^{\lambda |x|} } \leq 2\exp\sk{ - \lambda t + 16\lambda^2 \alpha^{-1}\| \Gamma(x, x) \|_\infty },
	\] 
	where the Chebyshev inequality is applied in the first inequality.
\end{proof}

\begin{remark}
	Additionally, we can derive a one-sided variant of the preceding corollary along with associated transportation inequalities, analogous to those established in \cite[Proposition 4.14, Theorem 4.18]{Junge2015}. We invite the interested reader to formulate the statements and to fill in technical details.
\end{remark}

\bigskip

\section{Non-tracial reduction}
\label{sec:extensions}

In this section, we extend our main $L_p$-Poincar\'e inequalities to general quantum Markov semigroups satisfyting GNS detailed balance in the non-tracial setting. The main tool is the Haagerup reduction argument devoloped in \cite{HJX}, and the proof follows from the similar arguments as in \cite[Section~4.6]{JW}.  The spectral gap hypothesis is imposed on the original semigroup. For the MLSI estimate, we require the same inequality on its dyadic crossed-product extension; a complete MLSI meets this hypothesis at the end of this section.

\subsection{GNS detailed balance and entropy}

% We now record the precise form in which the preceding tracial results pass to a non-tracial algebra.  
Let $\M$ be a $\sigma$-finite von Neumann algebra equipped with a normal faithful state $\phi$ and denote by $(\sigma^\phi_t)_{t \in \real}$ its corresponding modular automorphism group. A quantum Markov semigroup $(T_t)$ on $\M$ is said to satisfy GNS detailed balance with respect to $\phi$ if
\begin{equation}
 \phi((T_tx)^*y)=\phi(x^*T_ty),
 \quad x,y\in\M,\ t\geq0.
 \label{eq:GNS-detailed-balance}
\end{equation}
% This condition implies commutation with the modular automorphism group and is the compatibility needed to lift $(T_t)$ to the crossed product; see \cite[Section~4.2]{JW}.  
Let $A$ be the infinitesimal generator of $(T_t)$ on $\M$. Its fixed-point algebra $\M^T$ is invariant under $(\sigma^\phi_t)_{t \in \real}$,
and we denote by $E^\phi:\M \to \M^T$ the $\phi$-preserving conditional
expectation; see, e.g. \cite[Section~4.2]{JW} or \cite{GJLL2025}. On the Hilbert space of the GNS representation, we use $\|x\|_{2,\phi}=\phi(x^*x)^{1/2}$.

We present Haagerup $L_p$ space $L_p(\M, \phi)$ (or simply $L_p(\M)$), which is independent of the chosen state; see \cite{Terp1981}. Let $h_\phi$ be the Haagerup density of $\phi$.  Let $1\leq p\leq \infty$. Define the symmetric embedding operator
\begin{equation}
 \iota_p(x)=h_\phi^{1/(2p)}xh_\phi^{1/(2p)},
 \quad x\in\M,
 \label{eq:symmetric-Haagerup-embedding}
\end{equation}
and define the $L_p$ extensions by
\begin{equation}
 T_{t,p}\iota_p(x)=\iota_p(T_tx),
 \quad
 E_p^\phi\iota_p(x)=\iota_p(E^\phi x).
 \label{eq:Haagerup-extensions}
\end{equation}
Then $E_p^\phi$ extends to a contractive projection onto $L_p(\M^T)$, while $(T_{t, p})$ is a contraction semigroup on $L_p(\M)$; these facts follow from Kosaki's
interpolation and the identification with Haagerup $L_p$ spaces and Kosaki $L_p$ spaces \cite{Kosaki,HJX}.  

We say that $(T_t)$ satisfies $L_2$ spectral gap with constant $\lambda$ if
\begin{equation}
 \lambda\|x-E^\phi x\|_{2,\phi}^2
 \leq \phi(x^*Ax),
 \quad x\in {\rm dom}(A).
 \label{eq:GNS-spectral-gap}
\end{equation}
To define MLSI, we first introduce the Araki-Umegaki relative entropy \cite{Araki1976}. Note that there is an order-preserving isometric isomorphism
\[
\M_* \to L_1(\M), \quad \psi \mapsto h_\psi.
\]
Denote by $T_{t*}$ and $A_*$ respectively the pre-adjoints of $T_{t}$ and $A$ on $L_1(\M)$ instead of $\M_*$. For a normal faithful state $\sigma$ and $a\geq 1$, define
\[
B_a(\sigma) = \lk{ \psi \in \M_{*+} : a^{-1} \sigma \leq \psi \leq a \sigma }, \quad B(\sigma) = \bigcup_{a \geq 1} B_a (\sigma).
\]
Then $B(\sigma) \cap {\rm dom}(A_*)$ is dense in $\M_{*+}$, see \cite{Wirth2025}. For $\sigma, \psi \in \M_{*+}$, let $\Delta_{\psi, \sigma}$ denote the relative modular oprator and $e_{\psi, \sigma}$ its spectral measure. Define the Araki-Umegaki relative entropy by
\[
D(\psi \| \sigma) = \int_{(0, \infty)} \lambda \log \lambda \,\d \langle h_\sigma^{1/2}, e_{\psi, \sigma} h_\sigma^{1/2} \rangle + \infty \cdot \psi(1 - s(\sigma)),
\]
where we set the usual convention $\infty \cdot 0 = 0$ and $s(\sigma)$ is the support of $\sigma$. In particular, if $s(\sigma)=1$ and $\psi \in B_a(\sigma)$ for some $a \geq 1$, the log-difference $\log h_\psi - \log h_\sigma$ is bounded, and 
\[
D(\psi \| \sigma) = \psi(\log h_\psi - \log h_\sigma) = \tr( h_\psi( \log h_\psi - \log h_\sigma ) ),
\]
where ``$\tr$'' is the trace on the Haagerup space $L_1(\M)$. In addition, if $s(\sigma)=1$ and $\psi \in B(\sigma) \cap {\rm dom}(A_*)$, define the Fisher information by
\[
I_A(\psi \| \sigma) = A_*(\psi) ( \log h_\psi - \log h_\sigma ) = \tr( A_*(h_\psi)( \log h_\psi - \log h_\sigma ) ).
\]
For $\psi \in B(\phi) \cap {\rm dom}(A_*)$, by \cite[Theorem B.1]{Hiai2018} or \cite{Hiai2021}, one has
\begin{equation}\label{eq:entropy-chain}
D(\omega \| \phi) = D(\omega \| \psi) + \omega(\log h_{\psi} - \log h_{\phi}), \quad \omega \in \M_{*+}.
\end{equation} 
The identity is interpreted in the extended sense when an entropy is infinite.

Let $E^\phi_*$ be the pre-adjoint of $E^\phi$. We say that $(T_t)$ satisfies MLSI with constant $\alpha$ if
\begin{equation}\label{eq:state-MLSI}
      \alpha D(\psi \| E^\phi_*(\psi) ) \leq I_A(\psi \| E^\phi_*(\psi) ),
\end{equation}
for all $\psi \in {\rm dom}(A_*)$ with $E^\phi_*(\psi) \in B(\psi)$ faithful. In addition, $(T_t)$ satisfies complete MLSI (CMLSI) if \eqref{eq:state-MLSI} holds for $T_t \otimes {\rm id}_{\N}$ for all $\sigma$-finite von Neumann algebra $\N$.

The following two lemmas can be found in \cite{Wirth2025}.
\begin{lemma}\label{lem:Entropy-ce}
Let $\M \subseteq \widehat{\M}$ be a unital inclusion of von Neumann algebras with the normal faithful conditional expectation $E : \widehat{\M} \to \M$.
\begin{enumerate}[(i)]
      \item If $\psi \in \M_{*+}$, then $$D(\psi \circ E \| \phi \circ E ) = D(\psi \| \phi). $$
      \item If $\psi \in B(\phi)$, then $$E \sk{ \log h_{\psi \circ E} - \log h_{\phi \circ E} } = \log h_{\psi} - \log h_\phi.$$
\end{enumerate}
\end{lemma}

\begin{lemma}
Let $(\M_n)_{n \geq 1}$ be an increasing sequence of unital von Neumann subalgebras of $\M$ such that $\cup_{n \geq 1} \M_n$ is $w^*$ dense in $\M$. Let $\psi_n = \psi |_{\M_n}$ and $ \phi_n = \phi |_{\M_n}$.
\begin{enumerate}[(i)]
      \item If $\psi \in \M_{*+}$, then $$ D(\psi_n \| \phi_n ) \uparrow D(\psi \| \phi). $$
      \item If $\psi \in B(\phi)$, then $$ \log h_{\psi_n} - \log h_{\phi_n} \to \log h_{\psi} - \log h_{\phi} \quad \text{ in the $w^*$ topology}. $$
\end{enumerate}
\end{lemma}

\subsection{Haagerup's reduction}
Let $H$ be the Hilbert space of the GNS representation of $(\M, \phi)$ associated with $\phi$. Let
\[
 G=\bigcup_{m\geq1}2^{-m}\mathbb Z
\]
be equipped with the discrete topology. The crossed product
\[
 \widehat\M=\M\rtimes_{\sigma^\phi}G
\]
is the von Neumann algebra generated by the operators $\pi_{\sigma^\phi}(x)$, $x \in \M$ and $\lambda(g)$, $g \in G$, which are defined by
\[
(\pi_{\sigma^\phi}(x) \xi )(h) = \sigma^\phi_{-h}(x)\xi(h), \quad (\lambda(g)\xi)(h) = \xi(h - g), \quad \xi \in L_2(G; H), \ h \in G.
\]
Since $\pi_{\sigma^\phi}$ is a normal faithful representation of $\M$ on $L_2(G; H)$, we will identify $\M$ as $\pi_{\sigma^\phi}(\M)$ in the sequel.
Thus
$$ \lambda(g) x \lambda(g)^*=\sigma_g^\phi(x) , \quad x\in\M, \ g\in G . $$
There is a normal faithful conditional expectation
$E_\M:\widehat\M\to\M$ determined by
\[
 E_\M(x \lambda(g))=
 \begin{cases}
 x,&g=0,\\
 0,&g\neq0,
 \end{cases}
\]
and $\widehat\phi=\phi\circ E_\M$ is a normal faithful state on $\widehat\M$.

By \cite[Theorem 2.1]{HJX} and \cite[Section 4.2]{GJLL2025}, there exists an increasing sequence of finite
von Neumann subalgebras $(\M_m)_{m\geq1}$ of $\widehat\M$, and each $\M_m$ is equipped with a normal faithful tracial state $\tau_m$ such that
\begin{enumerate}
      \item $\cup_{m \geq 1} \M_m$ is $w^*$ dense in $\widehat{\M}$;
      \item For each $m \geq 1$, there is a normal faithful conditional expectation $\Phi_m:\widehat\M\to\M_m$ such that $\widehat\phi\circ\Phi_m=\widehat\phi$ and
\[
 \Phi_m\circ\sigma_s^{\widehat\phi}
 =\sigma_s^{\widehat\phi}\circ\Phi_m, \, \forall s\in \mathbb{R};
 \quad
 \Phi_m(x)\to x
 \ \text{$\sigma$-strongly for all } x\in \widehat{\M};
\]
      \item For each $m \geq 1$, there is a density $d_m \in \mathcal{L}(G) \subseteq \mathcal{Z}(\widehat{M}_{\widehat{\phi}})$ such that
\[
\tau_m(d_m \cdot) = \widehat{\phi}|_{\M_m}(\cdot) =: \phi_m(\cdot) .
\]
Furthermore, $a_m^{-1} \leq d_m \leq a_m$ for some $a_m \in \real_+$;
      \item Let $\psi$ be a state on $\M$, $\widehat{\psi}$ be its extension on $\widehat{\M}$ and $\psi_m = \widehat{\psi}|_{\M_m}$. Then $\psi_m \circ \Phi_m \to \widehat{\psi}$ in the weak topology on $\widehat{\M}$.
\end{enumerate}

The construction realizes $\M_m$ as the centralizer of a bounded
faithful normal positive functional whose modular group has period
$2^{-m}$.  $\Phi_m$ is obtained by averaging over one period, while restriction to the centralizer gives a finite trace on $\M_m$. The dyadic construction ensures that the centralizers are increasing and have weak$^*$-dense
union.  In particular, after the canonical isometric identifications,
\[
 L_p(\M)\subseteq L_p(\widehat\M),
 \quad
 \overline{\bigcup_{m\geq1}L_p(\M_m,\tau_m)}^{\,L_p}
 =L_p(\widehat\M),
 \quad 1\leq p<\infty.
\]
The induced maps $\Phi_{m,p}$ are contractions and satisfy
$\Phi_{m,p}(a)\to a$ in $L_p(\widehat\M)$; see \cite{HJX} for more details.
% Here the tracial state $\tau_m$ need not coincide with $\widehat\phi|_{\M_m}$.

The GNS detailed balance condition implies
$T_t\sigma_s^\phi=\sigma_s^\phi T_t$ for any $t\geq 0$ and $s\in\mathbb{R}$.  Hence $(T_t)$ and $E^\phi$
admit canonical extensions to $\widehat\M$, determined by
\[
 \widehat T_t  = T_t \otimes {\rm id}_{B(\ell_2(G))}|_{\widehat{\M}},
 \quad
 \widehat{E}^\phi = E^\phi \otimes {\rm id}_{B(\ell_2(G))}|_{\widehat{\M}}.
\]
Then the fixed-point subalgebra is $\widehat{\M}^{\widehat{T}} = \M^T \rtimes_{\sigma^\phi} G$. Moreover, $(\widehat{T}_t)$ is still a quantum Markov semigroup satisfying GNS detailed balance with respect to $\widehat{\phi}$; see \cite{GJLL2025}.
These extensions commute with every $\Phi_m$.  Moreover,
\[
 T_t^{(m)}=\widehat T_t|_{\M_m}
\]
is a $\tau_m$-symmetric quantum Markov semigroup, and
$E_m=\widehat{E}^\phi|_{\M_m}$ is its fixed-point expectation;
see \cite[Proposition~4.3]{HJX} and
\cite[Lemma 4.5]{GJLL2025}.  Thus, identifying
$x \in L_p(\M)$ with its image in $L_p(\widehat\M)$ and setting
$x_m=\Phi_{m,p}(x)$, we obtain
\[
 x_m-E_{m,p}(x_m)
 =\Phi_{m,p}(x-E_p^\phi (x))
 \to x-E_p^\phi (x)
 \quad\text{in }L_p(\widehat{\M}),
\]
where we use $ \widehat{E}^\phi \circ  \Phi_m= \Phi_m \circ \widehat{E}^\phi$ for all $m$; see \cite[Subsection 6.1]{HJX}.

\iffalse
The reduction argument is therefore to apply the tracial
$L_p$-Poincar\'e inequality to $a_m$ and then let $m\to\infty$.
Two additional ingredients are required: the gradients of the
approximants must converge in the sense specified below, and the
functional inequality used on the slices must have a constant
independent of $m$.  In the MLSI case, if every slice satisfies
\eqref{eq:mlsi} with the same constant $\alpha$, the tracial
coefficient
\[
 C_p(\alpha)
 =\frac{2^{3/2+1/p}(p-1)}{\sqrt{\alpha p}},
 \qquad p>2,
\]
passes to the limit without any additional factor.
The entropy approximation lemmas above do not by themselves establish
this uniform slice hypothesis from \eqref{eq:state-MLSI} on $\M$.
We formulate it separately as reduction-stability below.
\fi

Fix $2 \leq p < \infty$. To state the required gradient approximation precisely, we adopt the following  gradient-domain convention; see also \cite[Definition~4.13 and Remark~4.19]{JW}.
% The generator formula for $\Gamma_p(a,b)$ is used on $\mathfrak D_p$. For general self-adjoint elements of $L_p(\M)$, the gradient seminorm is understood through the lower semicontinuous relaxation described there. All Poincar\'e inequalities below are stated on the finite domain of this relaxed seminorm.
Denote by $A_p$ the generators of $(T_{t,p})$, and define
\[
 \mathfrak D_p=
 \{(x,y):x,y\in\Dom(A_p),\ x^*y\in\Dom(A_{p/2})\}.
\]
For $(x,y)\in\mathfrak D_p$, define the carr\'{e} du champ as
\begin{equation}
 \Gamma_p(x,y)=\frac12\bigl(
  A_p(x^*)y+x^*A_p(y)-A_{p/2}(x^*y)
  \bigr)\in L_{p/2}(\M).
 \label{eq:Haagerup-gamma}
\end{equation}
On the self-adjoint analytic core
\[
 \mathcal C_{p,\phi}=
 \{\iota_p(x):x\in {\rm dom}(A) \cap \M_{\rm sa} \text{ is entire}, \ (\iota_p(x), \iota_p(x)) \in \mathfrak D_p\},
\]
we define
\begin{equation}
 \| x \|_{\Gamma_p}=\|\Gamma_p(x, x) \|_{L_{p/2}(\M)}^{1/2}.
 \label{eq:Haagerup-gradient-seminorm}
\end{equation}
For a general self-adjoint $x\in L_p(\M)$, define the lower semicontinuous
extension by
\[
 \| x \|_{\Gamma_p}=
 \inf_{\substack{x_j\in\mathcal C_{p,\phi}\\x_j\to x\ {\rm in}\ L_p(\M)}}
 \liminf_{j\to\infty} \,\| x_j \|_{\Gamma_p},
\]
with the value $+\infty$ if no such sequence exists. 

\subsection{Regularization and $L_p$-Poincar\'{e} inequalities}
At first, we show a regularization lemma. For $\varepsilon>0$, define
\[
A_\varepsilon = A(I + \varepsilon A)^{-1} = \frac{1}{\varepsilon} \int_0^\infty e^{-s} (I - T_{\varepsilon s})\,\d s.
\]
Then $e^{-tA_\varepsilon}$ is a quantum Markov semigroup with the fixed-point subalgebra $\M^T$. Moreover, the extension $A_{\varepsilon, p}$ is bounded for $1\leq p<\infty$. For $2 \leq p < \infty$, define
\[
\Gamma_{\varepsilon, p}(x, y) = \frac{1}{2} \sk{ A_{\varepsilon, p}(x^*)y+x^*A_{\varepsilon, p}(y)-A_{\varepsilon, p/2}(x^*y) }
\]
for all $x, y \in L_p(\M)$. 

\begin{lemma}\label{lem:regular}
      If $(T_t)$ satisfies \eqref{eq:GNS-spectral-gap} with the constant $\lambda>0$, then $(e^{-tA_\varepsilon})$ also satisfies \eqref{eq:GNS-spectral-gap}, with the constant
      \[
      \lambda_\varepsilon \geq \frac{\lambda}{ 1 + \varepsilon \lambda }.
      \]
      If $(T_t)$ satisfies \eqref{eq:state-MLSI} with the constant $\alpha$, then  $(e^{-tA_\varepsilon})$ also satisfies \eqref{eq:state-MLSI}, with the constant
      \[
      \alpha_\varepsilon \geq \frac{\alpha}{ 1 + \varepsilon \alpha }.
      \]
\end{lemma}

\begin{proof}
We first deal with the case of spectral gap. It is known that \eqref{eq:GNS-spectral-gap} is equivalent to the following $L_2$-exponential decay: for self-adjoint $x \in {\rm dom}(A)$ with $E^\phi(x) = 0$,
\[
\phi( |T_t(x)|^2 ) \leq e^{-2\lambda t} \phi(x^2), \quad \forall\, t \geq 0.
\]
Since $(T_t)$ is GNS detailed balanced, by \eqref{eq:GNS-spectral-gap}, we obatin
\[
\begin{aligned}
\phi(x A_\varepsilon x) & = \frac{1}{\varepsilon} \int_0^\infty e^{-s} \sk{ \phi(x^2) - \phi(x T_{\varepsilon s}x) } \,\d s \\
& = \frac{1}{\varepsilon} \int_0^\infty e^{-s} \sk{ \phi(x^2) - \phi( | T_{\varepsilon s/2} (x) |^2 ) } \,\d s \\
& \geq \frac{1}{\varepsilon} \phi(x^2) \int_0^\infty e^{-s}(1 - e^{-\lambda \varepsilon s}) \,\d s \\
& = \frac{\lambda}{ 1 + \varepsilon \lambda } \phi(x^2).
\end{aligned} 
\]
This verifies \eqref{eq:GNS-spectral-gap} for $A_\varepsilon$.

As for the MLSI case, it follows from \cite[Theorem 2.7]{Wirth2025} that \eqref{eq:state-MLSI} is equivalent to: for any $\rho \in {\rm dom}(A_*)$ with $E^\phi_*(\rho) \in B(\rho)$ faithful,  
\begin{equation}\label{eq:entropy-decay}
      D( T_{t*}(\rho) \| E^\phi_*(\rho) ) \leq e^{-\alpha t} D(\rho \| E^\phi_*(\rho)), \quad \forall\, t \geq 0.
\end{equation}
Since $T_t$ is positive and $E^\phi_*(\rho) \in B(\rho)$, we have
\[
a^{-1} T_{t*}(\rho) \leq T_{t*}E^\phi_*(\rho) \leq a T_{t*}(\rho) 
\]
for some $a \geq 1$. Hence $E^\phi_*(\rho) \in B(T_{t*}(\rho))$. Let
\[
\sigma = (I - \varepsilon A_{\varepsilon *})(\rho) = \int_0^\infty e^{-s} T_{\varepsilon s *}(\rho) \,\d s. 
\]
By \eqref{eq:entropy-chain}, one has
\[
D(T_{\varepsilon s*}(\rho) \| E^\phi_*(\rho)) = D( T_{\varepsilon s*}(\rho) \| \sigma ) + T_{\varepsilon s*}(\rho)( \log h_\sigma - \log h_{E^\phi_*(\rho)} ). 
\]
By multiplying $e^{-s}$ to both sides and integrating in $s$, we obtain 
\begin{align*}
\int_0^\infty e^{-s} D( T_{\varepsilon s*}(\rho) \| E^\phi_*(\rho)) \,\d s & = \int_0^\infty e^{-s} D( T_{\varepsilon s*}(\rho) \| \sigma) \,\d s \\
& \qquad + \sigma( \log h_\sigma - \log h_{E^\phi_*(\rho)} ) \\
& \geq D(\sigma \| E^\phi_*(\rho)). 
\end{align*}
Therefore, by \eqref{eq:entropy-decay}
\begin{align}\label{eq:1-convex}
D(\sigma \| E^\phi_*(\rho)) \leq \int_0^\infty e^{-s} e^{-\alpha \varepsilon s} D( \rho \| E^\phi_*(\rho)) \,\d s = \frac{1}{1 + \varepsilon \alpha} D( \rho \| E^\phi_*(\rho)).
\end{align}
By \eqref{eq:entropy-chain} again,
\[
D(\sigma \| E^\phi_*(\rho)) = D(\sigma \| \rho) + \sigma ( \log h_\rho - \log h_{E^\phi_*(\rho)} ).
\]
This implies
\begin{align*}
      (\rho - \sigma)( \log h_\rho - \log h_{E^\phi_*(\rho)} ) & = D(\rho \| E^\phi_*(\rho)) - D(\sigma \| E^\phi_*(\rho)) + D(\sigma \| \rho) \\
      & \geq D(\rho \| E^\phi_*(\rho)) - D(\sigma \| E^\phi_*(\rho)).
\end{align*}
Plugging \eqref{eq:1-convex} into it, we obtain
\begin{align*}
I_{A_\varepsilon}(\rho \| E^\phi_*(\rho)) & = A_{\varepsilon *}(\rho) ( \log h_\rho - \log h_{E^\phi_*(\rho)} ) \\
& = \frac{1}{\varepsilon} (\rho - \sigma)( \log h_\rho - \log h_{E^\phi_*(\rho)} ) \\
& \geq \frac{1}{\varepsilon} (1 - \frac{1}{1 + \varepsilon \alpha})  D(\rho \| E^\phi_*(\rho)) \\
& = \frac{\alpha}{1 + \varepsilon \alpha} D(\rho \| E^\phi_*(\rho)).
\end{align*}
This completes the proof.
\end{proof}

We come to the following $L_p$-Poincar\'{e} inequalities for quantum Markov semigroups satisfying GNS detailed balance.

\begin{theorem}
\label{thm:nontracial-poincare}
Let $2 \leq p < \infty$. Let $(\M,\phi)$, $(T_t)$, and $E^\phi$ be as above. If $(T_t)$ satisfies \eqref{eq:GNS-spectral-gap} with the constant $\lambda > 0$, then for any self-ajoint $x \in L_p(\M)$,
\begin{equation}
 \|x - E_p^\phi (x)\|_{L_p(\M)}
 \leq
 \frac{\sqrt{2}p }{\sqrt\lambda}\,\| x \|_{\Gamma_p}.
 \label{eq:nontracial-gap-p}
\end{equation}
If the canonical extension $(\widehat{T}_t)$ satisfies \eqref{eq:state-MLSI} on $\widehat{\M}$ with the constant $\alpha > 0$, then for any self-ajoint $x \in L_p(\M)$,
\begin{equation}
 \|x - E_p^\phi (x)\|_{L_p(\M)}
 \leq
 \frac{ 2\sqrt{2p} }{\sqrt{\alpha}}\,\| x \|_{\Gamma_p}.
 \label{eq:nontracial-mlsi-p}
\end{equation}
\end{theorem}

\begin{proof}
The proof of the first assertion is as same as that in
\cite[Section~4.6]{JW}. We will only focus on the proof for the second assertion. It suffices to prove \eqref{eq:nontracial-mlsi-p} for $x \in \mathcal{C}_{p, \phi}$. The proof follows from the similar arguments as in \cite[Section~4.6]{JW}. We divide this proof into three steps.

\medskip

\noindent\emph{\textbf{Step 1}: $L_p$-Poincar\'{e} inequalities for $(T^{(m)}_{t, \varepsilon})$ on each $(\M_m, \tau_m)$.} Let $(\M_m, \tau_m)_{m \geq 1}$ be the increasing family of von Neumann subalgebras given as before. Fix $\varepsilon > 0$ and let $ \alpha_\varepsilon = \alpha/(1 + \varepsilon \alpha)$. Let $(T_{t, \varepsilon})$ be the quantum Markov semigroup generated by $A_\varepsilon$ and $\widehat{T}_{t, \varepsilon}$ be its extension to $\widehat{\M}$. By Lemma \ref{lem:regular}, $(\widehat{T}_{t, \varepsilon})$ satisfies \eqref{eq:state-MLSI} with the constant $\alpha_\varepsilon$. We show that each restriction $(T^{(m)}_{t, \varepsilon})$ also satisfies \eqref{eq:state-MLSI} on $(\M_m, \tau_m)$. 

Let $\Phi_{m *} : \M_{m*} \to \widehat{\M}_*$ be defined by $\Phi_{m *} (\rho) = \rho \circ \Phi_m$. We have
\[
\Phi_{m*} T_{t, \varepsilon *}^{(m)} = \widehat{T}_{t, \varepsilon *} \Phi_{m*}, \quad  \Phi_{m*} E_{m*} = \widehat{E}^\phi_{*} \Phi_{m*}.
\]
Then Lemma \ref{lem:Entropy-ce} and \eqref{eq:entropy-decay} yields that
\begin{align*}
D(T_{t, \varepsilon *}^{(m)} (\rho) \| E_{m*}(\rho) ) &
%  = D( x \circ T_{t, \varepsilon}^{(m)} \| x \circ E_m) \\ & = D( x \circ T_{t, \varepsilon}^{(m)} \circ \Phi_m \| x \circ E_m \circ \Phi_m ) \\  & = D( x \circ \Phi_m \circ \widehat{T}_{t, \varepsilon} \| x \circ \Phi_m \circ \widehat{E} ) \\
= D( \Phi_{m*} T_{t, \varepsilon *}^{(m)} (\rho) \| \Phi_{m*} E_{m*}(\rho) )\\
 & = D( \widehat{T}_{t, \varepsilon *} \Phi_{m*}(\rho) \| \widehat{E}^\phi_* \Phi_{m*}(\rho) ) \\
 & \leq e^{-t\alpha_\varepsilon }  D( \Phi_{m*}(\rho) \| \widehat{E}^\phi_* \Phi_{m*}(\rho) ) \\
 & = e^{-t\alpha_\varepsilon }  D( \rho \| E_{m*}(\rho) ) .
\end{align*}
Hence each $(T^{(m)}_{t, \varepsilon})$ also satisfies \eqref{eq:state-MLSI} with the constant $\alpha_\varepsilon$. Since $(\M_m, \tau_m)$ is tracial, it follows from Theorem \ref{Main-p-PI} that, for every self-adjoint $x \in L_p(\M_m, \tau_m)$,
\begin{equation}\label{eq:reduction-trace}
\| x - E_m(x) \|_{L_p(\M_m, \tau_m)} \leq \frac{ 2\sqrt{2 p } }{ \sqrt{\alpha_\varepsilon} } \| \Gamma_{\varepsilon}^{(m)} (x, x) \|_{L_{p/2}(\M_m ,\tau_m)}^{1/2},
\end{equation}
where $\Gamma_{\varepsilon}^{(m)}$ denotes the associated carr\'e du champ. 
%$A_\varepsilon^{(m)}$ denotes the generator of $T^{(m)}_{t, \varepsilon}$ and

\medskip

\noindent\emph{\textbf{Step 2}: $L_p$-Poincar\'{e} inequalities for $(T_{t, \varepsilon})$ on $\M$.}
We now employ the limiting argument. Note that the density $d_m \in \mathcal{L}(G) $ belongs to the fixed-point subalgera of $T^{(m)}_{t, \varepsilon}$. Then for any $u, v \in \real$ and $x\in \M_m$,
\[
T^{(m)}_{t, \varepsilon}(d_m^{u}xd_m^{v} ) = d_m^u T^{(m)}_{t, \varepsilon}(x) d_m^v, \quad E_m(d_m^u x d_m^v) = d_m^u E_m(x) d_m^v.
\]
The tracial and Haagerup $L_p$-norms are isometric, and hence for $x\in \M_m$
\[
\| d_m^{1/2p}xd_m^{1/2p} \|_{L_p(\M_m, \tau_m)} = \| \iota_p^{\phi_m}(x) \|_{L_p(\M_m, \phi_m)}.
\]
Recall that ($\sigma_t^{\phi_m}$) is the modular group of automorphisms with respect to $\phi_m$. Denote by $y_m = \sigma_{-i/2p}^{\phi_m}(x) = d_m^{1/2p} x d_m^{-1/2p}$ the modular shift. Then it is easy to verify
\[
\Gamma_{\varepsilon}^{(m)} (d_m^{1/2p}xd_m^{1/2p},  d_m^{1/2p}xd_m^{1/2p}) = d_m^{1/p} \Gamma_{\varepsilon}^{(m)} (y_m, y_m) d_m^{1/p}.
\]
Now let $\iota_p(x) \in \mathcal{C}_{p, \phi} \subseteq L_p(\M) \subseteq L_p(\widehat{\M})$. Then $x_m:= \Phi_m(x) \in \M_m$, and apply \eqref{eq:reduction-trace} for the element $ d_m^{1/2p} x_m d_m^{1/2p} \in L_p(\M_m, \tau_m) $,
\[
\| d_m^{1/2p} (x_m - E_m(x_m)) d_m^{1/2p} \|_{L_p(\M_m, \tau_m)}\leq \frac{ 2\sqrt{2 p } }{ \sqrt{\alpha_\varepsilon} } \| d_m^{1/p} \Gamma_{\varepsilon}^{(m)} (y_m, y_m) d_m^{1/p} \|_{L_{p/2}(\M_m ,\tau_m)}^{1/2}.
\]
That is
\[
\| \iota_p^{\phi_m}(x_m) - E_{m, p}(\iota_p^{\phi_m}(x_m)) \|_{L_p(\M_m, \phi_m)} \leq \frac{ 2\sqrt{2 p } }{ \sqrt{\alpha_\varepsilon} } \| \Gamma_{\varepsilon, p}^{(m)} (\iota_p^{\phi_m}(x_m), \iota_p^{\phi_m}(x_m)) \|_{L_{p/2}(\M_m ,\phi_m)}^{1/2}.
\]
Note that for $(x, y) \in \mathfrak{D}_p$, 
\begin{align*}
\| \Gamma_{\varepsilon, p}(x, y) \|_{p/2} & \leq \frac{1}{2} \sk{ \| A_{\varepsilon, p}(x)^* y \|_{p/2} + \| x^*A_{\varepsilon, p}(y) \|_{p/2} + \| A_{\varepsilon, p/2}(x^*y) \|_{p/2} } \\
& \leq \frac{3}{\varepsilon} \|x\|_p\|y\|_p < \infty,
\end{align*}
and then by letting $m \to \infty$, we obtain
\[
\| \iota_p(x) - E_p^\phi( \iota_p(x) ) \|_p \leq \frac{ 2\sqrt{2 p } }{ \sqrt{\alpha_\varepsilon} } \| \Gamma_{\varepsilon, p}( \iota_p(x), \iota_p(x) ) \|_{p/2}^{1/2},
\]
as desired.

\medskip

\noindent\emph{\textbf{Step 3}: $L_p$-Poincar\'{e} inequalities for $(T_{t})$ on $\M$.}
Since $\iota_p(x) \in \mathcal{C}_{p, \phi}$, we have $$ A_{\varepsilon, p}( \iota_p(x) ) \to A_p(\iota_p(x)),\quad \text{ in } L_p(\M),$$ 
and
$$ A_{\varepsilon, p/2}( \iota_{p}(x^*)\iota_{p}(x)) ) \to A_{p/2}( \iota_{p}(x^*)\iota_{p}(x)) ),\quad  \text{ in } L_{p/2}(\M). $$ 
Hence $$ \Gamma_{\varepsilon, p}( \iota_p(x), \iota_p(x) ) \to \Gamma_p(\iota_p(x), \iota_p(x)), \quad \text{ in } L_{p/2}(\M). $$ 
By letting $\varepsilon \to 0$, we obtain
\[
\| \iota_p(x) - E_p^\phi( \iota_p(x) ) \|_p \leq \frac{ 2\sqrt{2 p } }{ \sqrt{\alpha} } \| \Gamma_{p}( \iota_p(x), \iota_p(x) ) \|_{p/2}^{1/2} = \frac{ 2\sqrt{2 p } }{ \sqrt{\alpha} } \| \iota_p(x) \|_{\Gamma_p}.
\]

Finally, for general self-adjoint $x \in L_p(\M)$, if $x$ cannot be $L_p$-norm approximated by a sequence in $\mathcal{C}_{p, \phi}$, there is nothing to prove as $\| x \|_{\Gamma_p} = \infty$. Otherwise, we have for any $\delta > 0$, there exists a sequence $(x_j) \subseteq \mathcal{C}_{p, \phi}$ such that
\[
\lim_{j \to \infty} \| x_j - x \|_p = 0, \quad \liminf_{j \to \infty} \| x_j \|_{\Gamma_p} \leq \| x \|_{\Gamma_p} + \delta.
\]
Then
\begin{align*}
\| x - E_p^\phi(x) \|_p & \leq \liminf_{j \to \infty} \| x_j - E_p^\phi(x_j) \|_p \leq \frac{2\sqrt{2p}}{\sqrt{\alpha}} \liminf_{j \to \infty} \| x_j \|_{\Gamma_p} \\
& \leq  \frac{2\sqrt{2p}}{\sqrt{\alpha}} ( \| x \|_{\Gamma_p} + \delta ).   
\end{align*}
By letting $\delta \to 0$, this completes the proof.
\iffalse
If reduction-stable $\mathrm{MLSI}(\alpha)$ is assumed, every slice instead
satisfies Theorem~\ref{thm:mlsi-p} with the same $\alpha$.  The identical
limit argument proves \eqref{eq:nontracial-mlsi-p} and its endpoint on
$\mathcal C_{p,\phi}$.  For $p\geq5$, apply
Theorem~\ref{thm:sharp-high-p-mlsi} on every slice before taking the same
limit; this proves \eqref{eq:nontracial-sharp-high-p}.  The asymptotic
coefficient follows from \eqref{eq:M-p-kappa-asymptotic} and
\eqref{eq:c-zero-Dawson}.

Finally, let $a$ lie in the closed gradient domain and fix $\varepsilon>0$.
By the definition of $G_p^\phi$, there is a sequence
$a_j\in\mathcal C_{p,\phi}$ converging to $a$ in $L_p$ such that, after
passing to a subsequence,
\[
 \lim_{j\to\infty}G_{p,0}^\phi(a_j)
 \leq G_p^\phi(a)+\varepsilon.
\]
Contractivity of $E_p^\phi$ gives
$a_j-E_p^\phi a_j\to a-E_p^\phi a$ in $L_p$.  Applying the appropriate core
estimate, taking the limit, and then letting $\varepsilon\downarrow0$ proves
both assertions on the stated closed domains.
\fi
\end{proof}

We would like to remark that it still remains open whether the MLSI of $(T_t)$ implies the MLSI of $(\widehat{T}_t)$. However, we will show that if $(T_t)$ satisfies CMLSI, then  $(\widehat{T}_t)$ satisfies CMLSI as well.

% At the end of this section, we show that the CMLSI is sufficient for the assumption of our main theorem.

We first show the following auxiliary lemma. Let
\[
    F_n=
    \left\{
        j2^{-n}:j\in\mathbb Z,\,
        -2^{2n}\leq j\leq2^{2n}
    \right\} \subseteq G, \quad n \geq 1.
\]

\begin{lemma}\label{lem:folner-factorization}
There exists a sequence of normal unital completely positive (UCP) maps $\{u_n\}$ and $\{v_n\}$,
 \[ u_n : \widehat{\M} \to \M \,\overline{\otimes}\, B(\ell_2(F_n)), \quad v_n : \M \,\overline{\otimes} B(\ell_2(F_n))\, \to \widehat{\M}, \]
such that
\begin{enumerate}
    \item Suppose that $S: \M \to \M$ is a normal UCP map commuting with $\sigma^\phi_g$ for every $g\in G$.
Let $\widehat S$ be its extension to $\widehat{\M}$ and let $S_n=S \otimes \id_{B(\ell_2(F_n))}$.
Then
\begin{equation}\label{eq:ucp-intertwining}
    u_n\widehat S=S_n u_n,
    \quad
    v_nS_n= \widehat Sv_n.
\end{equation}
    \item Let $\Theta_n = v_nu_n$. Then for any $x\in\M$ and $g\in G$,
    \begin{equation}\label{eq:Theta-n}
    \Theta_n (x\lambda(g)) = a_n(g) x\lambda(g), \quad a_n(g) = \frac{ |F_n \cap (g + F_n)| }{|F_n|}\geq 0.
\end{equation}
    \item For any $\omega \in \widehat{\M}_*$, 
    \begin{equation}\label{eq:predual-approximation}
        \| \omega \circ \Theta_n - \omega \|_{\widehat{\M}_*} \to 0,
    \end{equation}
\end{enumerate}
\end{lemma}

\begin{proof}
We first realize the elements $x\lambda(g)$ in $\widehat{\M}$ as a matrix $( [x \lambda(g)]_{s, t} )_{s, t \in G}$. Note that for any $\xi \in L_2(G; H)$, $x \in \M$, $g \in G$,
\[
(x \lambda(g))\xi (s) = \sum_{t \in G} [ x\lambda(g) ]_{s, t} \xi(t) = \sigma^\phi_{-s}(x)\xi(s - g), \quad s \in G.
\]
Hence
\begin{equation}\label{eq:regular-matrix-entries}
    [ x\lambda(g) ]_{s, t}
    =
    \begin{cases}
        \sigma^\phi_{-s}(x),&s-t=g,\\
        0,&s-t\neq g.
    \end{cases}
\end{equation}
Let $p_n$ be the orthogonal projection from $L_2(G; H)$ onto $L_2(F_n; H)$.  Define
\[
u_n : \widehat{\M} \to \M \,\overline{\otimes}\, B(\ell_2(F_n)), \quad u_n(x \lambda(g)) = p_n x\lambda(g) p_n. 
\]
It is clear that $u_n$ is normal UCP. On the other hand, define
\[
v_n : \M \,\overline{\otimes}\, B(\ell_2(F_n)) \to \widehat{\M}, \quad v_n( ([x]_{s, t})_{s, t \in F_n} ) = \frac{1}{ |F_n| } \sum_{s, t \in F_n} \lambda(s) [x]_{s, t} \lambda(t)^*.
\]
It is also clear that $v_n$ is normal UCP.

For \eqref{eq:ucp-intertwining}, note that
\[
\widehat{S}(x\lambda(g)) = S(x)\lambda(g).
\]
Hence on matrix realization,
\begin{align*}
[u_n \widehat{S} ( x\lambda(g) )]_{s, t} & = [u_n( S(x)\lambda(g) )]_{s, t} = [ \sigma^\phi_{-s} S(x) \delta_{s-t, g}1_{\{s, t \in F_n\}} ]_{s, t} \\
& = [ S\sigma^\phi_{-s}(x) \delta_{s-t, g}1_{\{s, t \in F_n\}} ]_{s, t} = [S_n(u_n(x\lambda(g)))]_{s, t}.
\end{align*}
This proves the first identity of \eqref{eq:ucp-intertwining}. The other follows from a similar computation of $v_n$, and so we omit the details. Note that direct computation shows
\begin{align*}
    \Theta_n(x\lambda(g)) & = \frac{1}{|F_n|} \sum_{\substack{s,t\in F_n\\s-t=g}} \lambda(s) \sigma^\phi_{-s}(x)\lambda(t)^* \\
    & = \frac{1}{|F_n|} \sum_{\substack{s,t\in F_n\\s-t=g}} x \lambda(s-t) = \frac{ |F_n \cap (g + F_n)| }{|F_n|} x\lambda(g).
\end{align*}
This proves \eqref{eq:Theta-n}. It remains to show \eqref{eq:predual-approximation}. Fix $g= 2^{-m}k \in G$, where $k\in\mathbb Z$. Then, for $n$ sufficiently large,
\[
    a_n(g)
    =
    1-\frac{|k|2^{n-m}}{2^{2n+1}+1}
    \to 1, \quad \text{ as } n \to \infty.
\]
Therefore, for any finite Fourier expansion $f = \sum_{g \in G} \widehat{f}(g) x_g \lambda(g)$, by orthogonality,
\begin{align*} 
\| \Theta_n(f) - f\|_{2, \widehat{\phi}}^2 & = \sum_{g \in G} |a_n(g) - 1|^2|\widehat{f}(g)|^2 \phi(x_g^*x_g) \to 0.
\end{align*}
Since the maps $\Theta_n$ are contractions on the GNS Hilbert space, we conclude that
\begin{equation}\label{eq:gns-approximation}
    \|\Theta_n(f)- f \|_{2,\widehat\phi}
    \to 0,
    \quad \forall\, f\in  \widehat{\M}.
\end{equation}
The maps $\Theta_n$ are GNS symmetric, since $a_n(g)$ is real and $a_n(-g) = a_n(g)$. Moreover, for the normal functional $\omega_x(\cdot) := \widehat{\phi}(x^* \cdot)$, we have
\begin{align*}
  \| \omega_x \circ \Theta_n - \omega_x \|_{\widehat{\M}_*} & = \sup_{ \| y \|_{\widehat{\M}} \leq 1 } | \widehat{\phi}( x^*(\Theta_n(y) - y) ) |    \\
  & = \sup_{ \| y \|_{\widehat{\M}} \leq 1 } | \widehat{\phi} ((\Theta_n(x) - x)^*y ) | \leq \| \Theta_n(x) - x \|_{2, \widehat{\phi}}.
\end{align*}
Since $\{\omega_x: x \in \widehat{\M} \}$ is dense in $\widehat{\M}_*$, this implies \eqref{eq:predual-approximation}.
\end{proof}

\begin{remark}\label{rem:folner-factorization}
    Lemma \ref{lem:folner-factorization} holds for any $\phi$-preserving action $\eta : G \to {\rm Aut}(\M)$, with $\sigma^\phi |_{G}$ replaced by $\eta$ throughout the statement and proof. 
\end{remark}

\begin{proposition}
Let $(\M,\phi)$, $(T_t)$, and $E^\phi$ be as above. Let $(\widehat{T}_t)$ be the canonical extension of $(T_t)$ to $\M \rtimes_{\sigma^\phi} G$. If $(T_t)$ satisfies CMLSI with the constant $\alpha$, then $(\widehat{T}_t)$ satisfies CMLSI with the constant $\alpha$.
\end{proposition}

\begin{proof}
Fix a $\sigma$-finite von Neumann algebra $\mathcal N$
and a normal faithful state $\psi$ on $\mathcal N$.
Define
\[
    \widetilde{\M} = \M \,\overline{\otimes}\, \mathcal N,
    \quad
    \widetilde{\phi}=\phi\otimes\psi,
    \quad
    \eta_g=\sigma_g^\phi \otimes \id_{\mathcal N}.
\]
This action $\eta$ preserves $ \widetilde{\phi}$. There is a canonical normal $*$-isomorphism
\[
    \widetilde{\M} \rtimes_\eta G
    \cong \widehat{\M} \,\overline{\otimes}\, \mathcal N,
\]
Denote this algebra by $\widehat{\M}_{\mathcal N}$. Denote by
\[
\begin{aligned}
    S_t =T_t \otimes \id_{\mathcal N}, \quad 
    E_{\mathcal N} = E^\phi \otimes \id_{\mathcal N}.
\end{aligned}
\]
All $S_t$ and $E_{\mathcal N}$ are normal UCP commuting with $\eta_g$. We use the action version of Lemma~\ref{lem:folner-factorization}; see Remark~\ref{rem:folner-factorization}, since the action $\eta$ need not coincide with its modular action.

Define
\[
    S_{t, n}=S_t \otimes \id_{ B(\ell_2(F_n)) },
    \quad
    E_{\mathcal N, n} 
       =E_{\mathcal N} \otimes \id_{ B(\ell_2(F_n)) }.
\]
By Lemma~\ref{lem:folner-factorization}, there exist normal UCP maps $u_n$ and $v_n$ such that
\[
    u_n\widehat S_t = S_{t, n}u_n, \quad
    v_nS_{t, n} = \widehat S_t v_n,
\]
and
\[
    u_n\widehat E_{\mathcal N}
    =E_{\mathcal N, n} u_n, \quad
    v_nE_{\mathcal N, n}
       =\widehat E_{\mathcal N}v_n.
\]
Fix a normal state $\rho \in \widehat{\M}_{\mathcal \N*}$, it suffices to show 
\begin{equation}\label{eq:CMLSI-crossed}
      D( \widehat{S}_{t*}(\rho) \| \widehat{E}_{\N*}(\rho) ) \leq e^{-\alpha t} D( \rho \| \widehat{E}_{\N*}(\rho) ) .  
\end{equation}
Define normal states on $ \widetilde{\M} \,\overline{\otimes}\, B(\ell_2(F_n)) $ by $ \rho_n=\rho \circ v_n $.
Then
\begin{align}
    E_{\mathcal N, n *} (\rho_n) = \rho \circ v_n \circ E_{\mathcal N, n} =
    \rho \circ\widehat E_{\mathcal N}\circ v_n
    = \widehat{E}_{\N*}(\rho) \circ v_n.
    \label{eq:middle-reference-state}
\end{align}
Since $\mathcal N  \,\overline{\otimes}\, B(\ell_2(F_n))$ is $\sigma$-finite, and $(T_t)$ satisfies CMLSI, we obtain
\begin{equation}\label{eq:middle-cmlsi}
    D( S_{t, n*}(\rho_n) \| E_{\N, n*}(\rho_n) )
    \leq
    e^{-\alpha t}
    D(\rho_n \|  E_{\N, n*}(\rho_n)).
\end{equation}
By the data processing inequality,
\[
\begin{aligned}
D( \widehat S_{t*}(\rho)\circ\Theta_n \| \widehat E_{\mathcal N*}(\rho)\circ\Theta_n ) & = D( \rho \circ \widehat S_{t}\circ v_n \circ u_n \| \rho \circ \widehat E_{\mathcal N}\circ v_n \circ u_n ) 
\\
& \leq D( \rho \circ \widehat S_{t} \circ v_n \| \rho \circ \widehat E_{\mathcal N}\circ v_n) \\
& = D(S_{t, n*}(\rho_n) \| E_{\mathcal N, n*}(\rho_n)).
\end{aligned}
\]
On the other hand, by \eqref{eq:middle-reference-state},
\[
    D(\rho_n \| E_{\mathcal N, n*}(\rho_n))
    =
    D(\rho\circ v_n \| \widehat{E}_{\N*}(\rho) \circ v_n)
    \leq D( \rho \| \widehat{E}_{\N*}(\rho) ).
\]
Together with \eqref{eq:middle-cmlsi}, these estimates yield
\begin{equation}\label{eq:entropy-sandwich}
D( \widehat S_{t*}(\rho)\circ\Theta_n \| \widehat E_{\mathcal N*}(\rho)\circ\Theta_n ) \leq e^{-\alpha t}D( \rho \| \widehat{E}_{\N*}(\rho) ).
\end{equation}
Applying \eqref{eq:predual-approximation} to both normal states on the left-hand-side of \eqref{eq:entropy-sandwich}, and by the lower semicontinuity of relative entropy, we obtain
\[
\begin{aligned}
    D( \widehat{S}_{t*}(\rho) \| \widehat{E}_{\N*}(\rho) ) 
    &\leq
    \liminf_{n\to\infty}
    D( \widehat S_{t*}(\rho)\circ\Theta_n \| \widehat E_{\mathcal N*}(\rho)\circ\Theta_n ) 
    \\
    &\leq
    e^{-\alpha t}D( \rho \| \widehat{E}_{\N*}(\rho) ).
\end{aligned}
\]
This proves \eqref{eq:CMLSI-crossed}.
Since the auxiliary algebra $\mathcal N$ is arbitrary, $(\widehat T_t)$ satisfies CMLSI with the constant $\alpha$. 
\end{proof}

\bigskip

\section{Schatten-valued Poincar\'e inequalities on expanders}\label{sec:schatten-graphs}

This section is devoted to a concrete application of noncommutative $L_p$-Poincar\'e inequalities on the minimum distortion needed to embed any $n$-point metric space into Schatten class $S_p$. This solves an open question posed in \cite[Question 2]{EscuderoGutierrez2024}.  The key idea of the proof in this section is standard; for instance, see \cite{Matousek2002,Ostrovskii2013} and references therein. However, the new ingredient of our proof is just noncommutative $L_p$-Poincar\'e inequalities, i.e. Theorem \ref{Main-p-PI}.
% The problem is to compare average oscillation over all vertex pairs with average oscillation over edges for maps from a finite graph into a Schatten class, with a constant independent of the matrix dimension.  This comparison also produces metric-distortion lower bounds for expander families.

\subsection{Poincar\'{e} type estimate on graphs}
We first recall some basic notions on graphs. In this section, all graphs are assumed to be simple and undirected. Let $G=(V,E_G)$ be a finite connected $d$-regular simple undirected graph, where $d\geq1$. Its normalized adjacency matrix $Q=(Q_{xy})_{x, y\in V}$ is
\[
Q_{xy}=
\begin{cases}
	\dfrac{1}{d}, & x\sim y,\\[4pt]
	0, & \text{otherwise},
\end{cases}
\]
and the standard lazy walk is
\begin{equation*}
 P=\frac{I+Q}{2}.
 \label{eq:lazy-graph-walk}
\end{equation*}
Then the spectrum $\sigma(P)\subseteq [0,1]$. The Laplacian is defined by
\[
L = I - P.
\]
Let $V$ be equipped with the normalized counting measure, and $\mathsf E$ denote the corresponding expectation, namely, for $f : V \to \com$,
\[
\mathsf{E}(f)=|V|^{-1}\sum_{x\in V}f(x).
\]
Then $L$ generates a symmetric Markov semigroup $T_t = e^{-tL}$ on $L_\infty(V)$, with the fixed-point expectation $\mathsf{E}$. Observe that $P\mathsf{E} = \mathsf{E}$. Since $G$ is connected, the second largest spectrum of $P$ is
\begin{equation*}
 \rho(G)=\|P-\mathsf E\|_{L_2(V)\to L_2(V)} < 1.
\end{equation*}
Hence the spectral gap of $L$ is
$\delta(G) = 1 - \rho(G) > 0$.

For $1\leq p<\infty$, let $S_p$ be the Schatten class on a separable Hilbert space with norm $\|a\|_{S_p}=\operatorname{Tr}(|a|^p)^{1/p}$ for $a\in S_p$. For $f:V\to S_p$, define
\begin{align*}
 L_{V,p}(f)
 &=\left(\frac1{|V|^2}\sum_{x,y\in V}
        \|f(x)-f(y)\|_{S_p}^p\right)^{1/p},\\
 L_{E,p}(f)
 &=\left(\frac1{|E_G|}\sum_{\{x,y\}\in E_G}
        \|f(x)-f(y)\|_{S_p}^p\right)^{1/p}.
\end{align*}
Define
\begin{equation}
 \widetilde h_{S_p}(G)
 =\inf_{L_{V,p}(f)>0}\frac{L_{E,p}(f)}{L_{V,p}(f)}.
 \label{eq:Schatten-nonlinear-Poincare}
\end{equation}
The goal in this subsection is to estimate $\widetilde h_{S_p}(G)$ from noncommutative $L_p$-Poincar\'e inequalities; see Theorem \ref{thm:A-expander-schatten-poincare}. We proceed with the following elementary but quite helpful lemma.

\iffalse
Vertex, directed-edge, and undirected-edge averages use probability
normalization.  In particular, for symmetric quantities $b_{xy}=b_{yx}$,
\begin{equation}
 \frac1{d|V|}\sum_x\sum_{y\sim x}b_{xy}
 =\frac1{|E_G|}\sum_{\{x,y\}\in E_G}b_{xy}.
 \label{eq:directed-undirected-normalization}
\end{equation}
\fi
\begin{lemma}\label{lem:centered-pairwise-comparison}
Let $1\leq p<\infty$ and $f:V\to S_p$. Then
\begin{equation}
 \| f - \mathsf{E}(f) \|_{L_p(V; S_p)}
 \leq L_{V,p}(f)
 \leq2\| f - \mathsf{E}(f) \|_{L_p(V; S_p)}.
 \label{eq:centered-pairwise-comparison}
\end{equation}
\end{lemma}

\begin{proof}
For any fixed $x\in V$,
\[
 f(x) - \mathsf{E}(f)=|V|^{-1}\sum_{y\in V}(f(x)-f(y)).
\]
Then by the triangle inequality and the Jensen inequality, one has
\[
 \|f(x) - \mathsf{E}(f)\|_{S_p}^p
 \leq |V|^{-1}\sum_{y \in V}\|f(x)-f(y)\|_{S_p}^p.
\]
By taking expectaion on both sides in $x$, we obtain the first inequality. For the second, the triangle inequality yields
\[
 \|f(x)-f(y)\|_{S_p}^p
 \leq2^{p-1}\bigl(\| f(x) - \mathsf{E}(f) \|_{S_p}^p+
   \| f(y) - \mathsf{E}(f) \|_{S_p}^p\bigr).
\]
This completes the proof after taking expectation on both sides in two variables $x, y$. 
\end{proof}

\begin{theorem}\label{thm:A-expander-schatten-poincare}
Let $G$ be a finite connected $d$-regular graph. Then for $p \geq 2$,
\begin{equation}
 \widetilde h_{S_p}(G)
 \geq
 \frac{\sqrt{\delta(G)}}{ 2\sqrt{2}p },
 \label{eq:strong-expander-Schatten}
\end{equation}
and for $1 \leq p < 2$,
\begin{equation}
\label{eq:schatten-low-p-expansion}
    \widetilde h_{S_p}(G)
    \ge
    \frac{\delta(G)^{1/p}}{2^{1 + 1/p}}.
\end{equation}
\iffalse
At the Hilbert endpoint the stronger identity
\begin{equation}
 \widetilde h_{S_2}(G)=\sqrt{2\delta(G)}.
 \label{eq:Hilbert-expander-constant}
\end{equation}
\fi
Consequently, for every $1 \leq p < \infty$,
\begin{equation}
 \widetilde h_{S_p}(G) \geq  \frac{\delta(G)}{4p}.
 \label{eq:A-plus-and-A-prime}
\end{equation}
The constants are independent of both the matrix dimension and $|V|$.
\end{theorem}

\begin{proof}
We first consider finite-dimensional $S_p$-valued functions. Denote by
\[
 \M_G = L_\infty(V)\,\overline\otimes\,M_n(\com) \simeq L_\infty(V; M_n(\com)).
\]
Define the tracial state on $\M_G$ by
\[
 \tau_G(f)=\frac1{n|V|}\sum_{x\in V}\operatorname{Tr}(f(x)), \quad \forall f\in \M_G.
\]
Then the Laplacian, quantum Markov semigroup and corresponding fixed-point expectation on $\M_G$ are
\[ L \otimes {\rm id}_{M_n(\com)}, \quad T_t \otimes {\rm id}_{M_n(\com)}, \quad \mathsf E\otimes {\rm id}_{M_n(\com)}. \]
We still denote them by $L$, $T_t$ and $\mathsf{E}$ without causing any confusion. Thus $(T_t)$ has the spectral gap $\delta(G)$ on $L_2(\M_G,\tau_G)$. 
Let $\Gamma_G$ be the corresponding carr\'e du champ. Direct computation implies
\begin{equation}
 \Gamma_G(f,f)(x) = \frac1{4d}\sum_{y\sim x} |f(x) - f(y)|^2.
 \label{eq:graph-row-column-gamma}
\end{equation}

For $2 \leq p < \infty$, the map $X\mapsto\operatorname{Tr}(X^{p/2})$ is convex on
positive matrices. Therefore, we have
\begin{align*}
 \|\Gamma_G(f,f)^{1/2}\|_{L_p(\M_G)}^p & = \frac{1}{2^p n|V|} \sum_{x \in V} {\rm Tr}\sk{ \jdz{ \frac{1}{d}\sum_{y \sim x} |f(x) - f(y)|^2 }^{p/2} } \\
 & \leq  \frac{1}{2^p n|V|} \sum_{x \in V} \frac{1}{d}\sum_{y \sim x} {\rm Tr} (|f(x) - f(y)|^p)\\
 & = \frac1{2^pn|E_G|}
   \sum_{\{x,y\}\in E_G}\| f(x) - f(y) \|_{S_p}^p.
 \label{eq:column-by-edge}
\end{align*}
Hence, one has
\begin{equation}
 \|\Gamma_G(f, f)^{1/2}\|_{L_p(\M_G)}
 \leq\frac1{2n^{1/p}}L_{E,p}(f).
 \label{eq:rc-by-edge}
\end{equation}
Applying \eqref{main-1} in Theorem \ref{Main-p-PI} and Theorem \ref{thm:non-selfadjoint}, combining it with \eqref{eq:rc-by-edge}, we obtain
\begin{equation*}
\| f - \mathsf{E}(f) \|_{L_p(V; S_p)} = n^{1/p} \| f - \mathsf{E}(f) \|_{L_p(\M_G)}
 \leq 
 \frac{\sqrt{2}p}{\sqrt{\delta(G)}} L_{E,p}(f).
 \label{eq:centered-graph-Poincare}
\end{equation*}
Combining it with the second inequality in Lemma~\ref{lem:centered-pairwise-comparison} yields
\[
 L_{V,p}(f)
 \leq
 \frac{2\sqrt{2}p}{\sqrt{\delta(G)}}L_{E,p}(f).
\]
This proves \eqref{eq:strong-expander-Schatten} for $M_n(\com)$-valued functions.

For a general $S_p$-valued function, let $(e_n)$ be increasing finite-rank
projections converging strongly to the identity and let
$f_n(x)=e_nf(x)e_n$.  Since $V$ is finite,
$f_n(x)\to f(x)$ in $S_p$ simultaneously for all vertices.  Hence both
$L_{V,p}(f_n)$ and $L_{E,p}(f_n)$ converge to $L_{V, p}(f)$ and $L_{E, p}(f)$ respectively.
Passing to the limit yields \eqref{eq:strong-expander-Schatten} in $S_p$.

\iffalse
It remains to identify the endpoint.  For $S_2$-valued $f$ and
$g=f-\mathsf E_Vf$, expansion of the two quadratic forms gives
\begin{equation}
 L_{E,2}(f)^2=4\langle g,(I-P)g\rangle_{L_2(V;S_2)},
 \qquad
 L_{V,2}(f)^2=2\|g\|_{L_2(V;S_2)}^2.
 \label{eq:Hilbert-graph-identities}
\end{equation}
The Rayleigh quotient of $I-P$ on the centered subspace has infimum
$\delta(G)$.  Thus the quotient of the two expressions in
\eqref{eq:Hilbert-graph-identities} has infimum $2\delta(G)$.  A scalar
eigenfunction for the eigenvalue $\rho(G)$, multiplied by a rank-one element
of $S_2$, attains this infimum.  This proves
\eqref{eq:Hilbert-expander-constant}, which also implies
\eqref{eq:strong-expander-Schatten} when $p=2$.

For $p>2$, \eqref{eq:unified-Beckner-constant} gives
$C_p^{\mathrm B}\leq4p$.  Combining this with
$\sqrt{\delta(G)}\geq\delta(G)$ proves
\eqref{eq:A-plus-and-A-prime}; the case $p=2$ follows from
\eqref{eq:Hilbert-expander-constant}.
\fi

For $1 < p < 2$, for a self-adjoint matrix $x$ and $\theta>0$, write $ \Phi_\theta(x)=\operatorname{sgn}(x)|x|^\theta $ as the noncommutative Mazur map. It follows from
\cite[Lemma~2.6]{Ricard2015} that the following dimension-free estimate holds
\begin{equation}
\label{eq:mazur-low-p-holder}
    \|\Phi_{p/2}(x)-\Phi_{p/2}(y)\|_2
    \le
    2\|x-y\|_p^{p/2},
    \qquad 1\le p\le2,
\end{equation}
where $x, y$ are self-adjoint. Note that the constant $2$ is implicit in the original proof.

For any self-adjoint $f:V \to S_p$, let
\[
    g = f - \mathsf{E}(f),
    \quad
    h(x)=\Phi_{p/2}(g(x)), \ \forall x\in V.
\]
Thus
\begin{equation*}
    \mathsf{E}(g) = 0,
    \quad
    g(x) =\Phi_{2/p}(h(x)),
    \quad
    \|h(x)\|_{S_2}^2=\|g(x)\|_{S_p}^p.
\end{equation*}
By \eqref{eq:spectral-gap} and \eqref{eq:graph-row-column-gamma}, we have
\begin{equation}\label{eq:mazur-spectral-gap}
\delta(G) \| h - \mathsf{E}(h) \|_{L_2(V; S_2)}^2 \leq \frac{1}{4|E_G|} \sum_{ \lk{x, y}\in E_G } \| h(x) - h(y) \|_{S_2}^2.
\end{equation}  
On the other hand, \eqref{eq:mazur-low-p-holder} yields
\begin{equation}\label{eq:Mazur-p}
\|h(x) - h(y)\|_{S_2}^2
    \le
    4\|g(x)-g(y)\|_{S_p}^p
    =
    4\|f(x)-f(y)\|_{S_p}^p.
\end{equation}

Note that $\mathsf E ( \Phi_{2/p}(h) )=0$, we show
\begin{equation}
\label{eq:scalar-mean-defect}
    \|\mathsf E (h)\|_{S_2}^2\le\frac12 \| h \|_{L_2(V; S_2)}^2.
\end{equation}
Indeed, write $h=h_+-h_-$ and
\[
    U = \mathsf E (h_+^{2/p})=\mathsf E (h_-^{2/p}).
\]
Since $t \mapsto t^{p/2}$ is operator concave, by the Jensen inequality,
\[
    \mathsf{E}(h_{\pm}) = \mathsf{E} [ (h_{\pm}^{2/p})^{p/2} ] \leq [\mathsf{E}(h_{\pm}^{2/p})]^{p/2} = U^{p/2}.
\]
Then $-U^{p/2} \leq \mathsf{E}(h) \leq U^{p/2}$, and therefore
\[
\| \mathsf{E}(h) \|_{S_2}^2 \leq \| U^{p/2} \|_{S_2}^2. 
\]
Since $X \mapsto {\rm Tr}(X^{p})$ is convex, by the Jensen inequality,
\begin{align*}
\| U^{p/2} \|_{S_2}^2 & = \frac{1}{2} {\rm Tr} ( [\mathsf{E}(h_{+}^{2/p})]^p + [\mathsf{E}(h_{-}^{2/p})]^p ) \\
& \leq \frac{1}{2} \sk{ \mathsf{E} [{\rm Tr}( h_+^2 )] + \mathsf{E}[{\rm Tr} (h_-^2) ] } = \frac{1}{2}\| h \|_{L_2 (V; S_2)}^2.
\end{align*}
This proves \eqref{eq:scalar-mean-defect}. Hence
\[
\| h - \mathsf{E}(h) \|_{L_2(V; S_2)}^2 = \| h \|_{L_2(V; S_2)}^2 - \| \mathsf{E}(h) \|_{S_2}^2 \geq \frac{1}{2}\| h \|_{L_2(V; S_2)}^2
\]
Combining it with \eqref{eq:mazur-spectral-gap} and \eqref{eq:Mazur-p}, we obtain
\begin{align*}
\label{eq:low-p-centered-estimate}
    \| f - \mathsf{E}(f) \|_{L_p(V; S_p)}^p & = \frac{1}{|V|} \sum_{x \in V} \| g(x) \|_{S_p}^p =  \frac{1}{|V|} \sum_{x \in V} \| h(x) \|_{S_2}^2 \\ 
    & \le 2 \| h - \mathsf{E}(h) \|_{L_2(V; S_2)}^2 \leq
    \frac{2}{\delta(G)}L_{E,p}(f)^p.
\end{align*}
Finally, the second inequality in Lemma~\ref{lem:centered-pairwise-comparison} implies
\[
    L_{V,p}(f)
    \le 2 \| f - \mathsf{E}(f) \|_{L_p(V; S_p)}\leq \frac{2^{1 + 1/p}}{\delta(G)^{1/p}} L_{E, p}(f).
\]
As for an arbitrary matrix-valued $f$, apply the self-adjoint case to
\[
    \widehat f(x)=
    \begin{pmatrix}
        0&f(x)\\
        f(x)^*&0
    \end{pmatrix}.
\]
Note that
\[
    \|\widehat f(x)-\widehat f(y)\|_p^p
    =
    2\|f(x)-f(y)\|_p^p.
\]
\iffalse
both $L_{V,p}$ and $L_{E,p}$ are multiplied by $2^{1/p}$.
Thus the same constant holds for arbitrary matrices.
For maps into the infinite-dimensional Schatten class,
compress all values of $f$ by a common sequence of
finite-rank projections converging strongly to the identity.
These compressions converge in $S_p$ at every vertex.
Since $V$ is finite, passing to the limit proves
\eqref{eq:schatten-low-p-poincare} in full generality.
\fi
Thus the same inequality holds for arbitrary matrices. As a result, this proves \eqref{eq:schatten-low-p-expansion} for general $S_p$-valued functions.
\end{proof}

\subsection{Distortion of expanders}
For a finite metric space $(X,d_X)$, recall that the distortion $c_{S_p}(X)$ is the infimum of all
$D \geq1 $ for which there are $1$-Lipschitz maps $f:X\to S_p$ satisfying
\[
 \frac{1}{D} d_X(x,y)\leq\|f(x)-f(y)\|_{S_p}
 \leq d_X(x,y),
 \quad x,y\in X.
\]

The upper bound of the minimum distortion needed to embed any $n$-point metric space into Schatten class $S_p$ follows from the finite metric embedding estimate into $\ell_p$ shown in \cite{Matousek1997} and the isometric diagonal embedding $\ell_p \hookrightarrow S_p$. It suffices to show that this upper bound is optimal. We will see that Corollary \ref{cor:expander-distortion} provides such examples.

To construct such examples, we need to use expanders. We proceed with the introduction of graphs. A finite connected graph $G =(V, E_G)$ can be viewed as a metric space, equipped with the shortest-path metric, namely
\[
d_G(x, y) = \inf\lk{ n : x \sim x_1 \sim \cdots \sim x_{n - 1} \sim x_n = y }.
\]
Then we obtain the Schatten distortion of expanders by virtue of Theorem \ref{thm:A-expander-schatten-poincare}.
\begin{corollary}\label{cor:expander-distortion}
Let $(G_N = (V_N, E_{G_N}))$ be a family of regular graphs with $|V_N|=N$, degrees at most
$d_0$, and $\inf_N\delta(G_N)=\delta_0>0$, equipped with the shortest-path metric. For every $1 \leq p < \infty$,
\begin{equation}
 c_{S_p}(G_N) \approx \max\lk{1, \frac{\log N}{p} }.
 \label{eq:expander-Schatten-distortion}
\end{equation}

\end{corollary}

\begin{proof}
It suffices to show the lower bound as mentioned before. The case $d_0\leq2$ contains no infinite spectral expander family and needs no separate argument. For $d_0\geq3$, every ball of radius $r$ in a graph of maximum degree $d_0$
has at most
\[
 1+d_0\sum_{j=0}^{r-1}(d_0-1)^j = 1 + \frac{d_0}{d_0 - 2} [(d_0 - 1)^r-1]
\]
vertices. For a sufficiently large $N$, choose $ r_N=\lfloor c_{d_0}\log N\rfloor \geq c_{d_0}(\log N)/2 $
with $c_{d_0}>0$ small enough such that
\[
1 + \frac{d_0}{d_0 - 2} (d_0 - 1)^{r_N} \leq N/2. 
\]
Then, for each vertex $x$, there are at least $N/2$ vertices $y$ having distance $d_{G_N}(x, y) \geq r_N$, namely,
\[
|\lk{ y \in V_N: d_{G_N}(x, y) \geq r_N } | \geq N/2, \quad \forall\, x \in V_N.
\]
Consequently, for $p \geq 1$,
\begin{align}
 & \ \ \ \left(\frac1{|V_N|^2}\sum_{x,y\in V(G_N)}d_{G_N}(x,y)^p\right)^{1/p} \notag\\
 & \geq \left(\frac1{ |V_N|^2 }\sum_{ \substack{ x, y \in V_N \\ d_{G_N}(x, y) \geq r_N}  }d_{G_N}(x,y)^p\right)^{1/p} \notag\\
& \geq r_N N^{-2/p}|\lk{ (x, y) \in V_N \times V_N: d_{G_N}(x, y) \geq r_N } |^{1/p} \notag\\
& \geq  c_{d_0}'\log N.
 \label{eq:expander-average-distance}
\end{align}
Let $f:V_N \to S_p$ be an $1$-Lipschitz function with distortion $D$. Every edge has metric distance $1$, so 
\[ L_{E,p}(f) \leq \left(\frac1{|E_{G_N}|}\sum_{\{x,y\}\in E_{G_N}}
         |d_{G_N}(x, y)|^p \right)^{1/p} = 1. \]
On the other hand, by \eqref{eq:expander-average-distance},
\[
L_{V,p}(f) \geq \frac{1}{D}\left(\frac1{|V_N|^2}\sum_{x,y\in V(G_N)}d_{G_N}(x,y)^p\right)^{1/p} \geq \frac{c_{d_0}'}{D}\log N.
\]
Therefore Theorem \ref{thm:A-expander-schatten-poincare} yields, for $1 \leq p \leq \log N$,
\[
D \geq DL_{E,p}(f)
 \geq D\widetilde h_{S_p}(G_N)L_{V,p}(f)
 \geq c_{d_0}'
 \frac{ \delta_0 }{4p}\log N.
\]
Taking the infimum over all such $f$ yields the lower bound. Since the distortion is always at least $1$, the proof is complete.
\end{proof}

\begin{remark}
	The family of regular graphs with additional properties appearing in Corollary \ref{cor:expander-distortion} is called ``expander''. The notion of ``expander'' was discovered independently by Kolmogorov and Barzdin \cite{KB1967}, and Pinsker \cite{Pinsker1973}. However, the first explicit construction of expander was due to Margulis \cite{Margulis1973}.
	\end{remark}

\bigskip

\section{Examples}\label{sec:example}
This section is devoted to various applications to some specific quantum Markov semigroups and their best constants of $L_p$-Poincar\'{e} inequalities. We would like to mention that all specific examples of quantum Markov semigroups appearing in this section are ergodic. However, a simple way to construct non-ergodic quantum Markov semigroups is through amplification, namely, by tensoring with the identity map ${\rm id}_{M_n(\mathbb{C})}$. In this section, we focus on $L_p$-Poincar\'{e} inequalities only for ergodic quantum Markov semigroups. 

\subsection{Quantum tori}
Let $d \geq 2$ and $\theta = (\theta_{ij}) $ be a real skew symmetric $d \times d$ matrix. The $d$-dimensional noncommutative torus $\Aalg_\theta$ is the universal $C^*$-algebra generated by unitary operators $U_1, U_2, \cdots, U_d$ satisfying the commutation relations:
\[
	U_i U_j = e^{ 2\pi i \theta_{ij} } U_j U_i, \quad i, j = 1, 2, \cdots, d.
\]
Adopting Fourier series notation, for multi-indices $m = (m_1, m_2, \cdots, m_d) \in \ent^d$, a polynomial is defined as a finite sum
\[
	x = \sum_{ m \in \ent^d } \alpha_m U^m, \quad \alpha_m \in \com, \ U^m = U_1^{m_1} U_2^{m_2} \cdots U_d^{m_d},
\]
that is, $\alpha_m = 0$ for all but finite indices $m \in \ent^d$. The involution algebra $\mathcal P_\theta$ of all such polynomials is dense in $\Aalg_\theta$. The canonical tracial state is $\tau (x) = \alpha_{\bf 0}$, where ${\bf 0} = (0, 0 , \cdots, 0)$. Let $\mathbb T^d_\theta$ be the von Neumann algebra as the $w^\ast$-closure of $\Aalg_\theta$ in the GNS representation of $\tau$. In particular, when $\theta = 0$, $\Aalg_\theta=C(\mathbb{T}^d)$ and $\mathbb T^d_\theta=L_\infty(\mathbb{T}^d, \d m)$.

The heat semigroup $T_t$ on $\mathbb T_\theta^d$ is given by
\[
	T_t(U^m) = e^{-|m|^2t} U^m, \quad |m|^2:= \sum_{i = 1}^d |m_i|^2.
\]
The generator is the Laplacian $\Delta(U^m) = |m|^2U^m$. Direct computation leads to
\[
	\Gamma_T(U^m, U^n) = \sum_{j = 1}^dm_jn_j (U^{m})^*U^n, \ \Gamma_T^2(U^m, U^n) = \sum_{i,j = 1}^d m_im_jn_in_j (U^{m})^*U^n,
\]
for all $m , n \in \ent^d$.
Notably, $\Gamma_T^2 \geq 0$ but violates the $\Gamma^2$-condition. Nevertheless, $(T_t)$ is ultracontractive. Indeed, for any $x = \sum_{ m \in \ent^d } \alpha_m U^m \in L_2(\mathbb T^d_\theta)$
\begin{align*}
	\| T_t(x) \|_\infty \leq \sum_{m \in \ent^d} |\alpha_m|e^{-|m|^2 t} \leq \sk{ \sum_{m \in \ent^d} |\alpha_m|^2 }^{1/2}\sk{ \sum_{m \in \ent^d} e^{-2|m|^2t} }^{1/2} =: \phi(t)\| x \|_2. 
\end{align*}
The MLSI holds for $(T_t)$ with (not optimal) constant $\alpha_T = (2\log3)^{-1}$; see \cite[Corollary 6.4]{Brannan2023}. We also consider Poisson semigroups,	
\begin{align*}
	& P_t: \mathbb T_\theta^d \to \mathbb T^d_\theta, \quad P_t(U^m) = e^{-|m|t}U^m ; \\
	& Q_t: \mathbb T_\theta^d \to \mathbb T^d_\theta, \quad Q_t(U^m) = e^{-\|m\|_1t}U^m , \quad \| m \|_1 := \sum_{j=1}^{d}|m_j|,
\end{align*}
with associated carr\'{e} du champ operators denoted by $\Gamma_P$ and $\Gamma_Q$, respectively. According to \cite[Corollary 6.4]{Brannan2023}, the MLSI holds for both $(P_t)$ and $(Q_t)$ with one (not optimal) constant $\alpha_P = (2^{d+1}s_d (d-1)!)^{-1/d} $ (where $s_d = 2\pi^{d/2}/\Gamma(d/2)$ denotes the surface area of $(d-1) $-dimensional unit sphere) and the optimal constant $\alpha_Q = 2$, respectively. Therefore, Theorem \ref{Main-p-PI} immediately yields the following $L_p$-Poincar\'{e}-type inequalities.
\begin{corollary}
	Let $2 \leq p < \infty$. For any selfadjoint $x \in \mathcal P_\theta$, we have
	\[
		\| x - \tau(x) \|_p \leq 4 \sqrt{p\log 3} \| \Gamma_T(x, x)^{1/2} \|_{p},
	\]
	and
	\[
		\| x - \tau(x) \|_p \leq 4(2s_d (d-1)!)^{1/2d} \sqrt{p} \| \Gamma_P(x, x)^{1/2} \|_{p},
	\]
	\[
		\| x - \tau(x) \|_p \leq 2\sqrt{p} \| \Gamma_Q(x, x)^{1/2} \|_{p}.
	\]
\end{corollary}

We would like to stress that $(T_t)$ does not satisfy the $\Gamma^2$-condition. Hence, our Theorem \ref{Main-p-PI} complements Junge and Zeng's results \cite{Junge2015,Junge2015_2}, while the $\Gamma^2$-condition is essential in their proofs.

\subsection{Random matrices}
Let $\Omega$ be a Polish space equipped with a Borel probability measure
$\mu$, and let $\mathbb{E}$ denote the expectation. Let $(X_t)_{t\geq0}$ be a time-homogeneous Markov process with invariant measure $\mu$, and assume that its transition semigroup is ergodic,
$\mu$-reversible and strongly continuous on $L_2(\Omega,\mu)$.
We write $T_t=e^{-tA}$, where $A$ is nonnegative and self-adjoint so that
\[
    (T_t f)(x)
    =\mathbb E\!\left[f(X_t)\mid X_0=x\right]
\]
for $f\in L_2(\Omega,\mu)$ and $\mu$-almost every $x$.

\iffalse
We say that $(T_t^{(d)})$ satisfies scalar $L_2$-Poincar\'e inequality if for all $f: \Omega \to \com$ with $f \in {\rm dom}(A^{1/2})$,
\begin{equation}\label{eq:scalar-2P}
    \lambda \| f - \mathbb{E}(f) \|_2 \leq \Eform(f, f).
\end{equation}
In \cite[Proposition 2.4]{HT}, the authors proved that \eqref{eq:spectral-gap} for $(T_t^{(d)})$ is equivalent to \eqref{eq:scalar-2P} with the same constant.
\fi
To obtain estimates uniform in the matrix dimension, we use the following criterion for CMLSI in \cite[Theorem 3.9]{GJLL2025}.
\begin{lemma}\label{lem:7.2}
    Suppose that $(T_t)$ satisfies \eqref{eq:spectral-gap} with the constant $\lambda > 0$ and suppose that there exists $t_0 > 0$ and $C_0 \geq 1 $ such that
    \begin{equation}\label{eq:cb-time}
       \| T_{t_0} : L_1(\Omega, \mu) \to L_\infty(\Omega, \mu) \| \leq C_0. 
    \end{equation} 
    Then $(T_t)$ satisfies CMLSI with the constant
    \[
    \alpha \geq \frac{ \lambda }{\lambda t_0 + \log C_0 + 2}.
    \]
\end{lemma}

For $d \geq 1$, let $\M = M_d(\com)$ be the $d \times d$ complex matrix algebra equipped with the normalized trace $\operatorname{tr}_d$. For each $\mu$-integrable Borel function $f: \Omega \to \M$ and random variable $X \sim \mu$, we can define a random matrix $f(X)$. 
We identify bounded matrix-valued function space $L_\infty(\Omega; \M) \simeq L_\infty(\Omega, \mu) \overline{\otimes} \M$, equipped with the normalized trace $\mathbb{E} \otimes \operatorname{tr}_d$. For $1\leq p<\infty$, we use the notation
\[
    \|f\|_{L_p(\Omega;\M)}
    =
    \left(
        \mathbb{E}\operatorname{tr}_d(|f(x)|^p)
    \right)^{1/p}.
\]

For any $f \in L_\infty(\Omega; \M)$, we write $f=[f_{ij}]$, then
\[
    \mathbb E\!\left[f(X_t)\mid X_0=x\right]
    =
    \left[
        \mathbb E\!\left[f_{ij}(X_t)\mid X_0=x\right]
    \right]_{i,j=1}^d
    =
    \bigl[(T_t f_{ij})(x)\bigr]_{i,j=1}^d
\]
for $\mu$-almost every $x$.
Thus the transition operators acting on matrix-valued functions are
precisely the amplifications
\[
    T_t^{(d)}
    =T_t\otimes\operatorname{id}_{\M},
    \quad
    (T_t^{(d)}f)(x)
    =\mathbb E\!\left[f(X_t)\mid X_0=x\right].
\]
It is clear that $(T_t^{(d)})$ is a quantum Markov semigroup on $L_\infty(\Omega; \M)$ with the generator $A_d = A \otimes {\rm id}_\M$. The fixed point subalgebra is $\M$, and the corresponding conditional expectation is
$E = \mathbb{E} \otimes {\rm id}_{\M}$.
% We simly write $T_t$ for $T^{(d)}_t$.

Let $\Gamma$ be the carr\'e du champ of $(T_t)$. For suitable $f: \Omega \to \M $, direct computation shows that the carr\'e du champ of $(T^{(d)}_t)$ acts on $f = [f_{ij}]$ as
\[
\big[\Gamma^{(d)}(f, f) (x)_{ij} \big]= \Big[ \sum_{k = 1}^d \Gamma (f_{ki}, f_{kj}) (x) \Big], \quad x \in \Omega.
\]
In particular, $\Gamma^{(d)}(f,f)$ is positive.
Combining Lemma~\ref{lem:7.2} with
Theorem~\ref{Main-p-PI} implies the following estimate.

\begin{corollary}\label{coro:random mat}
Let $2 \leq p < \infty$. Suppose that the Markov semigroup $(T_t)$ satisfies \eqref{eq:spectral-gap} with $\lambda > 0$ and \eqref{eq:cb-time}. Then for any $f : \Omega \to \M_{\rm sa}$ with $f, f^2 \in {\rm dom}(A \otimes {\rm id}_\M)$, one has
\[
\| f - E(f) \|_{L_p(\Omega; \M)} \leq \frac{2 \sqrt{ 2(\lambda t_0 + \log C_0 + 2) p }}{\sqrt{\lambda}} \| \Gamma^{(d)}(f, f)^{1/2} \|_{L_{p}(\Omega; \M)}.
\]
\end{corollary}

Huang and Tropp \cite[Theorem~3.1]{HuangTropp2021} establish matrix-valued
$L_p$-Poincar\'e inequalities with $O(\sqrt p)$ growth under a positive scalar Bakry--\'Emery curvature bound.
Corollary~\ref{coro:random mat} gives the same order of growth from a spectral gap and \eqref{eq:cb-time}, but without imposing a curvature lower bound. Under the same assumptions as in Corollary \ref{coro:random mat}, Gao, Junge, LaRacuente and Li \cite[Theorem~5.12]{GJLL2025} obtained an $L_p$-Poincar\'e inequality controlled by
$\|\Gamma^{(d)}(f,f)^{1/2}\|_\infty$. Our estimate refines the gradient dependence of this bound by replacing
the $L_\infty$-norm with the corresponding $L_p$-norm.

\subsection{Mixed $q$-Gaussian algebras}
We recall the definition of mixed $q$-Gaussian algebras, cf. \cite{Bozejko1991, Bozejko1997, Bozejko1994,Lust-Piquard1999}. Let $H$ be a real Hilbert space with an orthonormal basis $\{e_j\}_{j \in J}$, and $Q = (q_{ij})_{i,j \in J}$ be a symmetric matrix satisfying $\sup_{i,j \in J} |q_{ij}| < 1$ (we can also consider the case $\sup |q_{ij}| \leq 1$ but with some necessary and appropriate modifications). For each $k \geq 1$, let $S_k$ be the permutation group of $\lk{1, \cdots, k}$, with adjacent transpositions $\sigma_j$ between $j$ and $j+1$ (for $k \geq 2$). Define the strictly positive definite twisted operators  (cf. \cite[Theorem 1.2]{Bozejko1994}) $P^{(k)}$: $H^{\otimes k} \to H^{\otimes k}$ for $k \geq 1$: let $u = e_{j_1} \otimes \cdots \otimes e_{j_k}$ be any $k$-basis atom, 
\[
    P^{(k)}(u) = \sum_{\sigma \in S_k} a(Q, \sigma, u) \sigma^{-1}(u)
\]
where $a(Q, \sigma, u)$ are coefficients determined by $Q$. More precisely,
\[
	a(Q, \sigma, u) = q_{j_{m_l} j_{m_{l+1}} } \prod_{i = 0}^{l-1} q_{j_{ \varphi_i(m_{l-i}) } j_{ \varphi_i(m_{l-i + 1}) } }, \quad \sigma = \sigma_{m_1} \cdots \sigma_{m_l},\ \varphi_i = \sigma_{m_{l- i + 1}} \cdots \sigma_{m_l},
\] 
with $a(Q, {\rm id}, u) = 1$. When $q_{ij} \equiv q$, this simplifies to
\[
    P^{(k)}(u) = \sum_{\sigma \in S_k} q^{I(\sigma)} \sigma(u), \quad I(\sigma) = \text{number of inversions in } \sigma.
\]

Let $\mathcal F_Q^{\text{finite}}(H) \subseteq \bigoplus_{k \geq 0} H^{\otimes k} $ be the subspace of finite sums, where $H^{\otimes 0} = \real\Omega$ and $\Omega$ is the vacuum vector. The \textit{mixed $q$-Fock space} $\mathcal F_Q(H)$ is defined as the completion of $\mathcal F_Q^{\text{finite}}(H)$ under the $Q$-deformed inner product
\[
    \langle \xi, \eta \rangle_Q = \delta_{kl} \langle \xi, P^{(l)} \eta \rangle, \quad \xi \in H^{\otimes k}, \eta \in H^{\otimes l},\ k, l \geq 0
\]
where $\lara{\cdot, \cdot}$ is the usual inner product on $H$. Define the left $r$-creation operator $l_i$ acting on $\mathcal F_Q(H)$ and its adjoint, $r$-annihilation operator $l_i^*$, as
\begin{align*}
    &l_i(\xi) = e_i \otimes \xi, \qquad l_i^*(\Omega) = 0, \\
    &l_i^*(e_{j_1} \otimes \cdots \otimes e_{j_k}) = \sum_{m=1}^k \delta_{i j_m} q_{j_m j_{m-1}} \cdots q_{j_m j_1} e_{j_1} \otimes \cdots \otimes \widehat{e_{j_m}} \otimes \cdots \otimes e_{j_k}
    \end{align*}
where $\widehat{e_{j_m}}$ denotes omission. These satisfy the deformed canonical commutation relations
\[
l_i^* l_j - q_{ij} l_j l_i^* = \delta_{ij} {\rm id}_{\mathcal F_Q(H)}.
\]
We define the \emph{$q$-Gaussian algebra} $\Gamma_Q(H)$ as the von Neumann algebra generated by the operators $\{s_i = l_i + l_i^* : i \in J\}$. It can be endowed with the canonical normal faithful tracial state 
$$\tau(x) = \langle \Omega, x\Omega \rangle_Q.$$ 

When $q_{ij} \equiv 1$, this is the Bosonic model for canonical commutation relations (CCR), and when $q_{ij} \equiv -1$, this is the Fermionic model for canonical anticommutation relations (CAR). When $q_{ij} \equiv 0$, this is Voiculescu's free Gaussian algebra \cite{Voiculescu1985}.

The \textit{mixed $q$-Ornstein-Uhlenbeck semigroup} $(T_t)_{t \geq 0}$ on $\Gamma_Q(H)$ is defined via second quantization. First, for a contraction $T : H \to H$, define $\mathcal F_Q(T) : \mathcal F_Q(H) \to \mathcal F_Q(H)$ by
\[
    \mathcal F_Q(T)\Omega = \Omega, \quad \mathcal F_Q(T)(f_1 \otimes \cdots \otimes f_k) = Tf_1 \otimes \cdots \otimes Tf_k.
\]
Recall that the map $\phi_H: \Gamma_Q(H) \to \mathcal F_Q(H), \, x \mapsto x \Omega$ extends to a unitary $\phi_H: L^2(\Gamma_Q(H), \tau) \to \mathcal F_Q(H)$, and note that $\phi_H(s_i) = e_i$. The second quantization $\Gamma_Q(T): \Gamma_Q(H) \to \Gamma_Q(H)$ is defined by
\[
    \Gamma_Q(T) = \phi_H^{-1} \circ \mathcal F_Q(T) \circ \phi_H.
\]
% We remark that $\Gamma_Q(T)$ is the unique completely positive map on $\Gamma_Q(H)$ satisfying the above formula, thus it is well-defined. Moreover, $\Gamma_Q$ is a functor, i.e. $\Gamma_Q(ST) = \Gamma_Q(S)\Gamma_Q(T)$ for two contractions $S, T$ on $H$. 
The \textit{mixed $q$-Ornstein-Uhlenbeck semigroup} is given by $T_t = \Gamma_Q(e^{-t} {\rm id}_H)$, $t \geq 0$. It is clear that the mixed $q$-Ornstein-Uhlenbeck semigroup is an ergodic quantum Markov semigroup. See \cite{Bozejko1991, Bozejko1997, Bozejko1994,Lust-Piquard1999} for more details.

% $(T_t)_{t \geq 0}$ extends to a semigroup of contractions on $L^2(\Gamma_R(H), \tau)$. It is symmetric and ergodic with generator $L = \phi_H^{-1} \circ N \circ \phi_H$ with $N: \F_R^{\text{finite}}(H) \to \F_R^{\text{finite}}(H)$ the number operator defined as $k {\rm id}$ on its eigenspace $H^{\otimes k}$.

Hypercontractivity for fermionic systems is originated in Gross' seminal work \cite{Gross1972, Gross1976}. After some partial results \cite{Gross1976, Lindsay1990, LindsayMeyer1992}, Carlen and Lieb achieved the optimal time bound \cite{CarlenLieb1993}. Then, these results were extended by Biane to $q$-Gaussians \cite{Biane1997}, subsequently generalized to free products by Junge \emph{et al.} \cite{Junge2015-HC}. 

For the MLSI side, we refer to \cite{Wirth2021,Brannan2022}, where the authors showed that $T_t$ satisfies a gradient estimate, implying the MLSI with constant $2$. Hence, Theorem \ref{Main-p-PI} yields
\begin{corollary}
	Let $2 \leq  p < \infty$. Denote by $\Gamma$ the carr\'{e} du champ of the mixed $q$-Ornstein-Uhlenbeck semigroup $(T_t)_{t \geq 0}$. Then
	\[
		\| x - \tau(x) \|_p \leq 2\sqrt{p} \| \Gamma(x, x)^{1/2} \|_{p}
	\]  
	holds for all suitable selfadjoint $x$.
\end{corollary}

Notably, Efraim and Lust-Piquard established the $L_p$-Poincar\'{e}-type inequalities with $O(\sqrt{p})$ constants for CAR \cite{Efraim2008}. Later, it was extended to $q$-Gaussians in \cite{Junge2015} with $O(p)$ constants. We recover and extend their results to the more general mixed $q$-Gaussian algebras.

\subsection{Group von Neumann algebras}
Let $G$ be a discrete group with identity $e$. The left regular representation $\lambda : G \to B(\ell_2(G))$ is given by $\lambda(g)\delta_h = \delta_{gh}$, where $\{\delta_g\}_{g \in G}$ is the canonical basis of $\ell_2(G)$. The \textit{group von Neumann algebra} $\mathcal{L}(G)$ is generated by $\text{span}\{\lambda(g) : g \in G\}$. It is endowed with a canonical faithful normal tracial state: for $f = \sum_g \widehat{f}(g)\lambda(g) \in \mathcal{L}(G)$
 $$\tau(f) = \langle \delta_e, f\delta_e \rangle = \widehat{f}(e). $$ 
A function $\psi : G \to \mathbb{R}_+$ is a \textit{conditionally negative definite length} (cn-length) if
\begin{enumerate}
    \item $\psi(e) = 0$ and $\psi(g) = \psi(g^{-1})$
    \item For any finite set $\{g_i\} \subseteq G$ and scalars $\{\xi_i\}$ with $\sum \xi_i = 0$:
    \[
    \sum_{i,j} \xi_i \overline{\xi_j} \psi(g_i^{-1}g_j) \leq 0.
    \]
\end{enumerate}
Given such $\psi$, we define a semigroup $(T_t)_{t \geq 0}$ on $\mathcal{L}(G)$ by
\[
T_t(\lambda(g)) = e^{-t\psi(g)}\lambda(g), \quad \forall \ g \in G.
\]
Its infinitesimal generator $A$ acts as $A\lambda(g) = \psi(g)\lambda(g)$. It is clear that $(T_t)$ is a quantum Markov semigroup. Let $H = \lk{g \in G: \psi(g) = 0}$, then the fixed-point subalgebra is $\M^T = \lambda(H)''$.

Let $\Aalg$ be the $*$-subalgebra of $\mathcal L(G)$ consisting of elements that can be written as finite combinations of $\lambda(g)$, $g \in G$. Then $\Aalg$ is weakly dense in $\mathcal L(G)$ such that $A (\Aalg) \subseteq \Aalg$ and $T_t(\Aalg) \subseteq \Aalg$. 
% In particular, $(T_t)_{t \geq 0}$ is ergodic if $\psi(g) = 0 \Leftrightarrow g = e$. 

The associated \textit{Gromov form} is
\[
K(g,h) = \frac{1}{2}\left( \psi(g) + \psi(h) - \psi(g^{-1}h) \right).
\]
For finitely supported $f = \sum_x \widehat{f}(x)\lambda(x)$, $g = \sum_y \widehat{g}(y)\lambda(y)$, the carr\'{e} du champ is 
\begin{align*}
\Gamma(f,g) &= \sum_{x,y} \overline{\widehat{f}(x)}\widehat{g}(y) K(x,y) \lambda(x^{-1}y).
%\Gamma^2(f,g) &= \sum_{x,y} \widehat{f}(x)\overline{\widehat{g}(y)} K(x,y)^2 \lambda(x^{-1}y).
\end{align*}
According to \cite[Theorem 3.3]{Brannan2023}, additional growth conditions are required to obtain the MLSI. 
\begin{lemma}\label{Growth-Cond}
Suppose
\begin{enumerate}[(i)]
\item (spectral gap) $\lambda = \inf_{g \notin H} \psi(g) > 0$,
\item (growth condition) $\exists\, r > 0$ such that $C_r = \sum_{g \notin H} r^{\psi(g)} < \infty$.
\end{enumerate}
Then $(T_t)_{t \geq 0}$ satisfies \eqref{eq:mlsi} with constant $\alpha = (2\lambda^{-1} \log(2C_r) - 2 \log r)^{-1}$.
\end{lemma}

% \begin{remark}
	% We remark that non-ergodic semigroups are concerned in \cite[Theorem 3.3]{Brannan2023} as well.
	%by replacing $g \neq e$ with $g \notin \M^T$ in both conditions, where $\M^T$ denotes the fixed point subspace.
% \end{remark}

\begin{remark}
	Hypercontractivity estimates were established in \cite[Theorem B]{Junge2017} under the ``Poisson-like'' length function assumption, which is stronger than the growth condition in Lemma \ref{Growth-Cond}.
\end{remark}

By Theorem \ref{Main-p-PI}, we establish the following $L_p$-Poincar\'{e}-type inequalities.
\begin{corollary}
	Let $2 \leq p < \infty$. Under the same assumptions in Lemma \ref{Growth-Cond},
	\[
		\| f - E(f) \|_p \leq C\sqrt{p} \| \Gamma(f, f)^{1/2} \|_{p}
	\]	
	holds for all selfadjoint $f \in \Aalg$, where $C = 2\sqrt{ 4\lambda^{-1} \log(2C_r) - 4\log r }$.
\end{corollary}

Now we present some typical group von Neumann algebras.
\subsubsection{Free group $\mathbb F_n$} 
Let $G = \mathbb{F}_n$ be the free group on $n$ generators with the word length $\psi(g) = |g|$. The Gromov form satisfies $K(g,h) = |\min(g,h)|:= \max\lk{|w|: g = wg', h = wh'}$ (length of longest common prefix subword). According to \cite{Haagerup1979}, $\psi$ is conditionally negative definite. The corresponding Markov semigroup is the free Poisson semigroup
\[
	P_t(f) = \sum_{g \in G} e^{-|g|t} \widehat{f}(g)\lambda(g).
\]
As verified in \cite[Example 3.6]{Brannan2023}, this semigroup satisfies the requisite growth conditions, implying the MLSI. Moreover, the hypercontractivity estimates were established in \cite{Junge2017,Ricard2016}. Thus it also satisfies $L_p$-Poincar\'{e}-type inequalities with $O(\sqrt{p})$ constants. While Junge and Zeng \cite{Junge2015_2} obtained identical estimates through the $\Gamma^2$-condition, our approach is different from theirs. 

\subsubsection{Walsh systems}
Let $\ent_n$ be the finite cyclic group and $\delta_{x, y}$ the Kronecker delta notation. Let $\mathbb{G} = \ent_n^m = \ent_n \times \cdots \times \ent_n$ be the cartesian product of $m$ copies of $\ent_n$ with length function
\[
	\psi(x_1, \cdots, x_m) = m - \delta_{x_1, 0} - \cdots - \delta_{x_m, 0}, \quad (x_1, \cdots, x_m)  \in \ent_n^m.
\]
This function admits spectral gap $\lambda = 1$ and exhibits growth constant $C_r = \sk{1 + r(n-1)}^m - 1 < \infty$. Consequently, the corresponding Markov semigroup satisfies $L_p$-Poincar\'{e}-type inequalities with $O(\sqrt{p})$ constants. We extend Efraim and Lust-Piquard's result \cite{Efraim2008} for Walsh systems $n=2$ to arbitrary $n$. Although Junge and Zeng \cite{Junge2015_2} derived comparable bounds via the $\Gamma^2$-condition, our framework offers a distinct methodological alternative.

\subsection{Compact quantum groups of Kac type}
We refer to \cite{Maes1998, Woronowicz1998} for fundamental theory of compact quantum groups. Let $(\mathbb{G}, \Delta)$ be a compact quantum group of \emph{Kac type}, meaning its Haar state is a trace (denoted by $\tau$ throughout). We consider quantum Markov semigroups on $L_\infty(\mathbb{G})$ symmetric with respect to $\tau$.

A quantum Markov semigroup $T_t: L_\infty(\mathbb{G}) \to L_\infty(\mathbb{G})$ is \emph{central} if for all $t \geq 0$, it is both left and right invariant
\[
\Delta \circ T_t = (T_t \otimes {\rm id}) \circ \Delta = ({\rm id} \otimes T_t) \circ \Delta,
\]
where $\Delta: L_\infty(\mathbb{G}) \to L_\infty(\mathbb{G}) \overline{\otimes} L_\infty(\mathbb{G})$ is the comultiplication.

Following \cite{Brannan2023}, we shall consider several specific central Markov semigroups which all satisfy the MLSI.

\subsubsection{Free Orthogonal Quantum Groups $O_N^+$}
The free orthogonal quantum groups \(O_N^+\) constitute a fundamental class of compact quantum groups introduced by Wang and Van Daele \cite{Wang1995,VanDaele1996}. For a fixed integer \(N \geq 2\), let \(O_N^+ = (C(O_N^+), \Delta)\) be the compact quantum group, where \(C(O_N^+)\) is the universal $C^*$-algebra generated by \(\{u_{ij}\}_{1 \leq i,j \leq N}\) such that \(u_{ij} = u_{ij}^*\) and the fundamental matrix \(u = [u_{ij}] = \sum_{i,j} u_{ij} \otimes e_{ij} \in C(O_N^+) \otimes M_N(\mathbb{C}) \) is unitary. The comultiplication \(\Delta: C(O_N^+) \to C(O_N^+) \otimes_{\text{min}} C(O_N^+)\) is given by  
\[
\Delta(u_{ij}) = \sum_{k=1}^N u_{ik} \otimes u_{kj}.
\]
Note that $u$ defined above is a unitary representation of $O_N^+$, called the fundamental representation. This quantum group admits a surjective $*$-homomorphism onto the classical orthogonal group \(C(O_N)\), realizing \(O_N\) as a quantum subgroup. 

The representation theory of \(O_N^+\), developed by Banica \cite{Banica1996,Banica1997}, exhibits remarkable structure. The set \(\operatorname{Irr}(O_N^+)\) of irreducible unitary representations is indexed by non-negative integers \(\mathbb{N}\), with fusion rules given by  
\[
k \otimes m \cong |k - m| \oplus (|k - m| + 2) \oplus \cdots \oplus (k + m), \quad k, m \in \mathbb{N}.
\] 
The trivial representation corresponds to \(k = 0\), while \(k = 1\) corresponds to the fundamental representation \(u\). The dimension \(n_k = \dim(u^k)\) satisfies the recurrence \(n_k N = n_{k+1} + n_{k-1}\) with initial conditions \(n_0 = 1\), \(n_1 = N\), and is explicitly given by \(n_k = U_k(N)\), where \(U_k\) are the Chebyshev polynomials of the second kind defined recursively via  
\[
U_0(x) = 1, \quad U_1(x) = x, \quad xU_k(x) = U_{k-1}(x) + U_{k+1}(x).
\]
Fixing an orthonormal basis for the representation space of \(u^k\), we obtain matrix coefficients \(u_{ij}^k \in \mathcal{O}(O_N^+) \subset C(O_N^+)\), and the Hopf $*$-algebra \(\mathcal{O}(O_N^+)\) of regular functions is linearly spanned by these coefficients. The Haar state \(\tau: C(O_N^+) \to \mathbb{C}\) is faithful on \(\mathcal{O}(O_N^+)\), and the associated GNS construction yields the von Neumann algebra \(L_\infty(O_N^+)\). Crucially, the matrix coefficients form an orthogonal basis for \(L_2(O_N^+)\) with  
\[
\tau\big((u_{i'j'}^{k'})^* u_{ij}^k\big) = \frac{\delta_{ii'}\delta_{jj'}\delta_{kk'}}{n_k}.
\]
Subsequently, we adopt a slight abuse of notation, identifying $\mathcal{O}(O_N^+)$ with subspaces of $L_\infty(O_N^+)$ and $L_2(O_N^+)$ through the conventional GNS constructions.

Central Markov semigroups on \(O_N^+\) are completely determined by their action on matrix coefficients. Specifically, as in \cite{Cipriani2014,Franz2017}, the \textit{heat semigroup} $T_t = e^{-tA}$ on \(O_N^+\) is given by its generator acting diagonally on irreducible representations
\[
	A: \mathcal{O}(O_N^+) \to \mathcal{O}(O_N^+);\  A(u_{ij}^k) = \lambda_k u_{ij}^k, \quad \lambda_k  = \frac{U_k'(N)}{U_k(N)} \in \mathbb{C}.
\]
It corresponds to the quantum Laplace-Beltrami operator. Then $T_t$ extends uniquely to an ergodic central quantum Markov semigroup $T_t : L_\infty(O_N^+) \to L_\infty(O_N^+)$. Explicitly, key estimates for its eigenvalues and dimensions are given in \cite[Lemma 1.7]{Franz2017} by  
\[
n_k \leq N^k, \quad \frac{k}{N} \leq \lambda_k \leq \frac{k}{N-2}, \quad k \geq 0.
\]

The hypercontractivity estimates and $2$-LSI with $O(N)$ constants were established in \cite{Franz2017, Brannan2021}, with related estimates for optimal time bounds. For the MLSI side, Brannan \emph{et al.} \cite{Brannan2023} investigated the complete MLSI. Specifically, they established the MLSI for $T_t \otimes {\rm id}_{\N}$, where $\N$ is any finite von Neumann algebra, using geometric curvature lower bounds and estimates of the complete bounded return time. Their derived constant is quantified as $\alpha = (2N\log\sk{ \sqrt{3}N / (\sqrt{3} - \sqrt{2}) } )^{-1}$. 
\begin{corollary}
Let $\Gamma$ denote the carr\'{e} du champ associated with $T_t$. For all $2 \leq p < \infty$ and selfadjoint element $x \in \mathcal{O}(O_N^+)$, the following inequality holds
\[
	\| x - \tau(x) \|_p \leq C\sqrt{p} \| \Gamma(x, x)^{1/2} \|_{p}.
\]
Here, $C = 4\sqrt{ N\log\sk{ \sqrt{3}N / (\sqrt{3} - \sqrt{2}) } }$.
\end{corollary}

\subsubsection{Quantum Automophism Groups}
The quantum automorphism groups provide another significant class of compact quantum groups that generalize the symmetries of finite-dimensional C*-algebras. Given a finite-dimensional C*-algebra $B$ equipped with its canonical tracial state $\tau$ (obtained by restricting the unique normalized trace on $\mathrm{End}(B)$ through the left-regular representation $B \hookrightarrow \mathrm{End}(B)$), we define the \textit{quantum automorphism group} of the pair $(B,\tau)$, denoted by $G^+(B,\psi)$. To construct this quantum group, let $H = L^2(B,\tau)$ and fix an orthonormal basis $\{e_i\}$ for $H$, with corresponding matrix units $e_{ij} \in B(H)$. The universal C*-algebra $C^u(G^+(B,\tau))$ is generated by elements $\{u_{ij}\}$ satisfying the following fundamental relations:
\begin{enumerate}
    \item The fundamental matrix $u = [u_{ij}] = \sum_{i,j} u_{ij} \otimes e_{ij} \in C^u(G^+(B,\tau)) \otimes B(H)$ is unitary;
    \item The coaction preserves multiplication: $(1 \otimes m)u_{12}u_{13} = u(1 \otimes m)$ where $m: H \otimes H \to H$ is the multiplication map of $B \cong H$;
    \item The coaction preserves the unit: $u(1 \otimes \eta) = 1 \otimes \eta$ where $\eta: \mathbb{C} \to B \cong H$ is the unit map $\alpha \mapsto \alpha 1_B$.
\end{enumerate}
The comultiplication $\Delta: C^u(G^+(B,\tau)) \to C^u(G^+(B,\tau)) \otimes_{\text{min}} C^u(G^+(B,\tau))$ is defined on generators by $\Delta(u_{ij}) = \sum_k u_{ik} \otimes u_{kj}$, making $G^+(B,\psi) = (C^u(G^+(B,\psi)), \Delta)$ a compact quantum group.

A particularly important special case occurs when $B = C(X)$ for a finite set $X$ with $N$ elements and $\tau$ is the uniform probability measure. Here, $G^+(B,\tau)$ reduces to the \textit{quantum permutation group} $S_N^+$ \cite{Wang1995}, where the generators satisfy the additional relations $u_{ij} = u_{ij}^* = u_{ij}^2$ and $\sum_k u_{ik} = 1 = \sum_k u_{ki}$ for all $1 \leq i \leq N$. Crucially, while $S_N^+$ coincides with the classical symmetric group $S_N$ for $N \leq 3$, it becomes infinite-dimensional and noncommutative for $N \geq 4$.

For general $B$ with $\dim B \geq 4$, the representation theory of $G^+(B,\tau)$ parallels that of $O_N^+$. The irreducible representations $\operatorname{Irr}(G^+(B,\tau))$ are indexed by non-negative integers $\mathbb{N}$, with fusion rules
\[
k \otimes m \cong |k-m| \oplus (|k-m|+1) \oplus \cdots \oplus (k+m-1) \oplus (k+m), \quad k,m \in \mathbb{N}.
\]
% The trivial representation corresponds to $k=0$, while the fundamental representation $u$ corresponds to the label $1$. 
Let $v_{ij}^k$ be the generic coefficient of the $k$-th irreducible representation $v^k$ of $ G^+(B,\tau) $ and $\mathcal{O} \sk{ G^+(B,\tau) } $ be the Hopf algebra spanned by such matrix coefficients. The dimensions $m_k = \dim(v^k)$ satisfy $m_k = U_{2k}(\sqrt{d})$ where $d = \dim B$ and $U_k$ are the Chebyshev polynomials.

Analogous to $O_N^+$, the \textit{heat semigroup} on $G^+(B,\tau)$ (with $d = \dim B \geq 4$) is defined as $T_t = e^{-tA} : L_\infty( G^+(B,\tau) ) \to L_\infty(G^+(B,\tau))$ given by (see \cite{Cipriani2014})
\[
	T_t \sk{v_{ij}^k} = e^{-\xi_k t} v_{ij}^k,\ A \sk{v_{ij}^k} = \xi_k v_{ij}^k, \ \xi_k = \frac{U_{2k}'(\sqrt{d})}{2\sqrt{d}  U_{2k}(\sqrt{d})}.
\]
This semigroup is ergodic and satisfies dimension and eigenvalue estimates
\[
m_k \leq (d-1)^k, \quad \frac{k}{d} \leq \xi_k \leq \frac{k}{\sqrt{d}(\sqrt{d} - 2)}, \quad k \geq 0.
\]

The hypercontractivity estimate for quantum permutation group $S_N^+$ was established in \cite{Franz2017}. For general quantum automophism groups, Brannan \emph{et al.} \cite{Brannan2023} obtain the MLSI with constant quantified as $\alpha = ( 2d \log\mk{ 3(d-1)/(4 - \sqrt{13}) } )^{-1}$. Then, Theorem \ref{Main-p-PI} yields the following $L_p$-Poincar\'{e}-type inequalities with $O(\sqrt{p})$ constants.

\begin{corollary}
	Let $\Gamma$ denote the carr\'{e} du champ associated with $T_t$. For any $2 \leq p < \infty$ and any selfadjoint element $x \in \mathcal{O}(G^+(B, \tau))$, the following inequality holds
	\[
		\| x - \tau(x) \|_p \leq C\sqrt{p} \| \Gamma(x, x)^{1/2} \|_{p}.
	\]
	Here, $C = 4\sqrt{ d \log\mk{ 3(d-1)/(4 - \sqrt{13}) } }$.
\end{corollary}

It seems that our results on $L_p$-Poincar\'{e}-type inequalities for heat semigroups associated with and \(O_N^+\) and $G^+(B,\tau)$ are even new.

\bigskip
{\bf AI statement.} The authors obtained Theorems 1.1, 3.4, and 4.4 in the ergodic setting independently, without any assistance from AI tools in 2025/06. Through subsequent discussions with ChatGPT-5.6 Sol, the authors developed proofs of these theorems in the non-ergodic setting. AI tools were also used to assist with LaTeX editing and language polishing. The authors have independently verified the mathematical arguments and take full responsibility for the correctness of all results presented in this paper.

{\textbf{Acknowledgments.}} 
Li Gao is supported in part by the National Natural Science Foundation of China (grant no. 12401163) and the Department of Science and Technology of Hubei Province (project nos. 2025EHA041 and 2025AFA044). 
% Marius Junge was partially supported by NSF Grant DMS-2247114. 
Zhendong Xu is partially supported by the National Research Foundation of Korea (NRF) grant funded by the Korea government (MSIT) (No.2020R1C1C1A01009681), Samsung Science and Technology Foundation under Project Number SSTF-BA2002-01 and the Basic Science Research Program through the National Research Foundation of Korea (NRF) (Grant RS-2022-NR069971).

\end{document}